\documentclass[a4paper,12pt, reqno]{amsart}

\usepackage[a4paper, left=0.85in, right=0.85in, top=0.85in, bottom=0.85in]{geometry}

\usepackage{textcomp}
\usepackage[no-math]{fontspec}
\usepackage[english]{babel}

\usepackage{mathtools}
\usepackage{amssymb, amsfonts}
\usepackage{thmtools}
\usepackage{bm}
\usepackage{bbm}
\usepackage{mathdots}
\usepackage{csquotes}
\usepackage{graphicx}
\usepackage{caption}

\usepackage{enumitem}
\setlist[enumerate,1]{label={\upshape(\arabic*)}}

\usepackage[scr=pxtx]{mathalpha}
\usepackage[hidelinks]{hyperref}

\DeclareSymbolFontAlphabet{\mathbb}{AMSb}

\newcommand\Z{\ensuremath{\mathbb{Z}}}

\newcommand\Q{\ensuremath{\mathbb{Q}}}

\DeclareMathOperator{\Homeo}{Homeo}
\DeclareMathOperator{\Homeop}{\ensuremath{\Homeo^{+}}}

\DeclareMathOperator{\Mod}{Mod}

\DeclareMathOperator{\I}{\mathcal{I}}
\DeclareMathOperator{\K}{\mathcal{K}}

\DeclareMathOperator{\A}{\mathcal{A}}
\DeclareMathOperator{\B}{\mathbb{B}}
\DeclareMathOperator{\C}{\mathcal{C}}

\DeclareMathOperator{\PP}{\mathcal{P}}

\DeclareMathOperator{\RR}{\mathcal{R}}
\DeclareMathOperator{\FF}{\mathcal{F}}

\DeclareMathOperator{\Int}{Int}
\DeclareMathOperator{\Sp}{Sp}

\DeclareMathOperator{\Arf}{Arf}

\DeclareMathOperator{\sgn}{sgn}
\DeclareMathOperator{\im}{im}

\let\AA\fobarrel
\let\SS\fobarrelsdf

\DeclareMathOperator{\AA}{\mathscr{A}}
\DeclareMathOperator{\SS}{\mathscr{S}}
\DeclareMathOperator{\CC}{\mathscr{C}}

\newcommand{\Symp}{\ensuremath{\mathbf{Symp}_2}}

\let\1\foobarbaz
\let\0\foobarbaz

\newcommand{\ba}{\ensuremath{\bm{\mathrm{a}}}}
\newcommand{\bb}{\ensuremath{\bm{\mathrm{b}}}}

\newcommand{\be}{\ensuremath{\bm{\mathrm{e}}}}

\newcommand{\bx}{\ensuremath{\bm{\mathrm{x}}}}
\newcommand{\by}{\ensuremath{\bm{\mathrm{y}}}}
\newcommand{\bz}{\ensuremath{\bm{\mathrm{z}}}}
\newcommand{\bU}{\ensuremath{\bm{\mathrm{U}}}}
\newcommand{\bX}{\ensuremath{\bm{\mathrm{X}}}}
\newcommand{\bY}{\ensuremath{\bm{\mathrm{Y}}}}
\newcommand{\bV}{\ensuremath{\bm{\mathrm{V}}}}
\newcommand{\bW}{\ensuremath{\bm{\mathrm{W}}}}

\newcommand{\0}{\ensuremath{\bm{0}}}
\newcommand{\1}{\ensuremath{\bm{1}}}

\let\Cap\foobarbar

\newcommand{\Cap}{\ensuremath{\mathcal{C}\!\!\!\!\:\,\,\mathit{ap}}}

\usepackage{tikz}
\usetikzlibrary{
    shapes.geometric,
    calc,
    arrows.meta,
    decorations.markings,
    math,
    cd,
    babel
}

\usepackage{pgfplots}
\pgfplotsset{compat=1.18}
\usepgfplotslibrary{polar}

\makeatletter
\newcommand{\proofoutline}[1]{%
  \par
  \medskip
  \noindent
  {\itshape Outline of the proof of #1.}\enspace\ignorespaces
}
\makeatother

\numberwithin{equation}{section}

\newtheorem{maintheorem}{Theorem}

\newtheorem{theorem}{Theorem}[section]
\newtheorem{proposition}[theorem]{Proposition}
\newtheorem{lemma}[theorem]{Lemma}
\newtheorem{corollary}[theorem]{Corollary}

\newtheorem*{fact*}{Fact}

\theoremstyle{definition}

\newtheorem{remark}[theorem]{Remark}
\newtheorem{question}[theorem]{Question}

\usepackage{import}
\usepackage{float}
\usepackage{xifthen}
\usepackage{pdfpages}
\usepackage{transparent}

\usepackage[
    backend=biber,
    language=autobib,
    autolang=other,
    sorting=nyt,
    giveninits=true, 
    doi=false,
    url=false,
    isbn=false
]{biblatex}

\renewbibmacro*{volume+number+eid}{%
    \printfield{volume}%
    \setunit{\addcolon}%
    \printfield{number}%
    \setunit{\addcomma\space}%
    \printfield{eid}%
}

\renewbibmacro{in:}{}
\title{Torsion in the homology of the Torelli group and
the Birman--Craggs--Johnson homomorphism}
\author{Andrei Vladimirov}
\address{Lomonosov Moscow State University, Russia}
\keywords{Mapping class group, Torelli group, Birman--Craggs--Johnson
homomorphism, abelian cycles, homology of groups}

\subjclass{57M07 (Primary); 20J05; 20J06 (Secondary)}

\email{andreykaere1@gmail.com}
\date{}

\begin{document}
    \begin{abstract}
        The \emph{Birman--Craggs--Johnson homomorphism} is a homomorphism
        \(\sigma \colon \I_g \to \B_3'\) from the Torelli group to a certain
        \(\Z/2\Z\)-vector space of Boolean polynomials. In 1983, Johnson computed
        \(H_1(\I_g)\) for \(g \geq 3\) and showed, in particular, that the induced
        homomorphism on \(H_1(\I_g)\) is injective when
        restricted to the subgroup generated by Dehn twists about separating
        simple closed curves.
        In this paper, we extend Johnson's result to higher homology groups.
        Given any collection of pairwise disjoint separating simple closed
        curves on \(\Sigma_g\), the corresponding Dehn twists pairwise commute
        and determine a homology class in \(H_k(\I_g)\) called an
        \emph{abelian cycle}. We prove that the pushforward homomorphism
        restricted to the subgroup of \(H_k(\I_g)\) generated by such abelian
        cycles is injective for \(k \leq g-2\).
    \end{abstract}

    \maketitle
    \markboth{}{}

    \section{Introduction}

Let \(\Sigma_g^b\) denote a compact surface of genus \(g\) with \(b\) boundary
components; when the surface is closed, we omit the superscript and write
\(\Sigma_g\). Recall that the \emph{mapping class group} of \(\Sigma_g^b\) is
\[
\Mod(\Sigma_g^b) = \pi_0 (\Homeop(\Sigma_g^b, \partial \Sigma_g^b))
,\] 
where \(\Homeop(\Sigma_g^b, \partial \Sigma_g^b)\) is the group of
orientation-preserving homeomorphisms of \(\Sigma_g^b\) that fix each boundary
component pointwise.

The Torelli group \(\I_g\) is the kernel of the surjective homomorphism
\(\Mod(\Sigma_g) \to \Sp_{2g}(\Z)\) arising from the action of
\(\Mod(\Sigma_g)\) on \(H_1(\Sigma_g)\). Similarly, one can define the Torelli
group \(\I_g^1\) for a surface with one boundary component. It is well-known
that \(\I_1\) is trivial, and Mess~\cite{mess} showed that \(\I_2\) is an
infinitely generated free group. Later Johnson~\cite{johnson1}
proved that \(\I_g\) and \(\I_g^1\) are finitely generated whenever \(g \geq 3\).

Studying the homology of the Torelli group is a natural and fundamental
problem in geometric topology. This investigation originated in 1980s,
when Johnson~\cite{johnson3} famously calculated \(H_1(\I_g)\). In later
decades, Bestvina, Bux, and Margalit showed that for \(g \geq 2\), the group
\(\I_g\) has cohomological dimension \(3g-5\) and that the top homology group
\(H_{3g-5}(\I_g)\) is not finitely generated. Gaifullin~\cite{gaifullin-inf}
further extended this result, proving that the groups \(H_k(\I_g)\) are also
infinitely generated for \(2g-3 \leq k \leq 3g-6\). 

Recently, Minahan and Putman~\cite{minahan-putman} proved that \(H_2(\I_g;
\Q)\) is finite-dimensional for \(g \geq 5\) and is an algebraic
representation of \(\Sp_{2g}(\Z)\) for \(g \geq 6\). Their result implies that
the calculation of the algebraic part \(H_2(\I_g; \Q)^{\mathrm{alg}}\) of the
representation of \(\Sp_{2g}(\Z)\) by Kupers and
Randal-Williams~\cite{kupers-williams} coincides with \(H_2(\I_g; \Q)\) for
\(g \geq 6\). This, in turn, highlights the importance
of studying the torsion in \(H_k(\I_g)\), which is the central focus of the
present work.

\subsection{The abelian cycles}

To state the main results of the paper, we first recall the definition of the
abelian cycle. Suppose that \(h_1, \dots, h_k \in G\) pairwise commute.
Consider the homomorphism \(\phi \colon \Z^k \to G\) given by sending
generator of the \(i^{\text{th}}\) factor to \(h_i\). Then the
\emph{abelian cycle} \(\A(h_1, \dots, h_k)\) is defined to be the image of the
standard generator \(\mu \in H_k(\Z^k) \cong \Z\) under the pushforward
homomorphism \(\phi_{*} \colon H_k(\Z^k) \to H_k(G)\).

Let us recall the following standard properties of abelian cycles that we will
use freely throughout the paper. These are a direct consequence of the
identifications \(H_k(\Z^k) \cong \wedge^k \Z^k \cong \Z\).

\begin{enumerate}[label={(\Roman*)}]
    \item Let \(h_1', h_1, \dots, h_k \in G\). In \(H_k(G)\), we have
    \[
    \A(h_1 h_1', h_2, \dots, h_k) = \A(h_1, h_2, \dots, h_k) + \A(h_1', h_2, \dots, h_k)
    ,\]
    whenever all three abelian cycles are defined.

    \item Let \(h_1, \dots, h_k \in G\). For any permutation \(\pi\)
    of \(\{1, \dots, k\}\), we have
    \[
    \A(h_{\pi(1)}, \dots, h_{\pi(k)}) = \sgn(\pi) \A(h_1, \dots, h_k)
    .\]
    
    \item If \(x, y \in G\) commute with \(z_1, \dots, z_{n-1} \in G\), then
    \[
    \A([x, y], z_1, \dots, z_{n-1}) = 0
    ,\]
    where \([x, y] = x y x^{-1} y^{-1}\) is the commutator of \(x\) and \(y\).
\end{enumerate}
\medskip

\subsection{Main results}

Let \(T_{\delta}\) denote the left Dehn twist about a simple closed curve
\(\delta\). We consider the subgroup \(H_k^{\mathrm{ab,sep}}(\I_g) \subset
H_k(\I_g)\) generated by the abelian cycles of the form \(\A(T_{\delta_1},
\dots, T_{\delta_k})\), where \(\delta_1, \dots, \delta_k\) are pairwise
disjoint, separating simple closed curves on \(\Sigma_g\).

Determining the structure of \(H_k^{\mathrm{ab,sep}}(\I_g)\) is a natural and
important problem; in particular, whether or not it is  finitely generated.
In~\cite{vladimirov}, the author proved that:

\begin{enumerate}[label={(\roman*)}]
    \item The group \(H_k^{\mathrm{ab,sep}}(\I_g)\) is a \(\Z/2\Z\)-vector
    space for \(k \geq 2\) and \(g \geq 3\).

    \item For \(k \leq g-1\), the space \(H_k^{\mathrm{ab,sep}}(\I_g)\) is
    generated by cycles \(\A(T_{\delta_1}, \dots, T_{\delta_k})\) where each
    curve \(\delta_i\) has genus~\(1\), and \(H_k^{\mathrm{ab,sep}}(\I_g)\) is
    trivial for \(k \geq g\).

    \item The space \(H_2^{\mathrm{ab,sep}}(\I_g)\) is finite-dimensional for
    \(g \geq 4\).
\end{enumerate}

In the present paper, we extend result (iii) to higher
homology groups.

\begin{maintheorem} \label{mainthm:fin-dim}
    The \(\Z/2\Z\)-vector space \(H_k^{\mathrm{ab,sep}}(\I_g)\) is
    finite-dimensional for \(g \geq 4\) and \(2 \leq k \leq g-2\).
\end{maintheorem}

In view of statement (ii), the following question naturally arises.

\begin{question}
    Is the \(\Z/2\Z\)-vector space \(H_{g-1}^{\mathrm{ab,sep}}(\I_g)\)
    finite-dimensional?
\end{question}

In 1980, Johnson~\cite{johnson-bcj}, using results of Birman and
Craggs~\cite{birman-craggs}, constructed an \(\Sp_{2g}(\Z)\)-equivariant
homomorphism \(\sigma \colon \I_g \to \B_3'\), where \(\B_3'\) is a specific
\(\Z/2\Z\)-vector space of Boolean polynomials; see
Section~\ref{sec:preliminaries} for a precise definition.

The pushforward homomorphism \(\sigma_{*} \colon H_k(\I_g) \to \wedge^k
\B_3'\) acts on abelian cycles by
\[
\sigma_{*}(\A(T_{\delta_1}, \dots, T_{\delta_k})) = \sigma(T_{\delta_1}) \wedge \dots \wedge \sigma(T_{\delta_k})
.\]
Since \(\sigma(T_{\delta}) \in \B_2'\) for any separating twist
(see~\eqref{eq:BCJ-sep} in Section~\ref{sec:preliminaries}), the image of
\(\sigma_{*}\) restricted to the subgroup \(H_k^{\mathrm{ab,sep}}(\I_g)\) is
contained in \(\wedge^k \B_2'\). We therefore obtain a homomorphism
\[
\sigma_k \colon H_k^{\mathrm{ab,sep}}(\I_g) \to \wedge^k \B_2'
.\]

In~\cite{johnson3}, Johnson proved that the Birman--Craggs--Johnson
homomorphism \(\sigma \colon \I_g \to \B_3'\) induces an isomorphism of
\(\Sp_{2g}(\Z)\)-modules \(H_1^{\mathrm{ab,sep}}(\I_g) \cong \B_2'\) for \(g
\geq 3\). In the present paper, we extend Johnson's result as follows.

\begin{maintheorem} \label{mainthm:BCJ-inj}
    For the Birman--Craggs--Johnson homomorphism \(\sigma \colon \I_g \to
    \B_3'\), the pushforward \(\Sp_{2g}(\Z)\)-equivariant homomorphism
    \[
    \sigma_k \colon H_k^{\mathrm{ab,sep}}(\I_g) \to \wedge^k \B_2'
    \]
    is injective for \(g \geq 4\) and \(2 \leq k \leq g-2\).
\end{maintheorem}

\begin{remark}
    In particular, Theorem~\ref{mainthm:BCJ-inj} yields Theorem~1.7
    of~\cite{vladimirov} as a direct consequence.
\end{remark}

\begin{corollary}
    Since \(\dim \B_2' = 2 g^2 + g\), Theorem~\ref{mainthm:BCJ-inj} yields the
    upper bound
    \[
    \dim H_k^{\mathrm{ab,sep}}(\I_g) \leq \binom{2 g^2 + g}{k}
    .\]
\end{corollary}

As another corollary to Theorem~\ref{mainthm:BCJ-inj}, we provide a complete
set of relations for the abelian cycles of the form \(\A(T_{\delta_1}, \dots,
T_{\delta_k}) \in H_k(\I_g)\).

\begin{corollary}
    Every relation among abelian cycles of the form \(\A(T_{\delta_1}, \dots,
    T_{\delta_k}) \in H_k(\I_g)\), where \(\delta_1, \dots, \delta_k\) are
    pairwise disjoint separating simple closed curves, follows from the
    following four families of relations:
    \begin{enumerate}
        \item \(\A(T_{\delta_{\pi(1)}}, \dots, T_{\delta_{\pi(k)}}) =
        \A(T_{\delta_1}, \dots, T_{\delta_k})\) for any permutation \(\pi\) of
        \(\{1, \dots, k\}\);
        \item \(\A(T_{\delta_1}, \dots, T_{\delta_k}) = 0\) if two of the
        curves \(\delta_1, \dots, \delta_k\) are isotopic;
        \item \(2 \A(T_{\delta_1}, \dots, T_{\delta_k}) = 0\);
        \item \(\A(T_{\gamma_1}, T_{\delta_2}, \dots, T_{\delta_k}) + \dots +
        \A(T_{\gamma_n}, T_{\delta_2}, \dots, T_{\delta_k}) = 0\) whenever
        \(\sigma(T_{\gamma_1}) + \dots + \sigma(T_{\gamma_n}) = 0\), where
        \(\gamma_1, \dots, \gamma_n\) are separating simple closed curves
        disjoint from \(\delta_2 \cup \dots \cup \delta_k\).
    \end{enumerate}
\end{corollary}

\begin{proof}
    By Theorem~\ref{mainthm:BCJ-inj}, \(\sigma_k\) induces an isomorphism
    \(H_k^{\mathrm{ab,sep}}(\I_g) \cong \im \sigma_k \subset \wedge^k \B_2'\).
    Thus, the stated relations clearly hold, and it suffices to show that any
    relation in \(\im \sigma_k\) pulls back to them.

    The image \(\im \sigma_k\) is the quotient of the free abelian group on
    generators \(f_1 \otimes \dots \otimes f_k\) (where \(f_i =
    \sigma(T_{\delta_i})\)) modulo the relations:
    \begin{enumerate}[label={(\alph*)}]
        \item \(f_1 \otimes \dots \otimes f_k = f_{\pi(1)} \otimes \dots \otimes f_{\pi(k)}\) for any permutation \(\pi\);
        \item \(f_1 \otimes \dots \otimes f_k = 0\) if \(f_i = f_j\) for some \(i \neq j\);
        \item \(2 (f_1 \otimes \dots \otimes f_k) = 0\);
        \item \(p \otimes f_2 \otimes \dots \otimes f_k + q \otimes f_2 \otimes
        \dots \otimes f_k = (p+q) \otimes f_2 \otimes \dots \otimes f_k\), where
        \(p\) and \(q\) are of the form \(\sum_{i} \sigma(T_{\gamma_i})\) for curves
        \(\gamma_i\) disjoint from \(\delta_2 \cup \dots \cup \delta_k\);
        \item \(p \otimes f_2 \otimes \dots \otimes f_k = 0\), where \(p = \sum_{i}
        \sigma(T_{\gamma_i}) = 0\) for curves \(\gamma_i\) disjoint from
        \(\delta_2 \cup \dots \cup \delta_k\).
    \end{enumerate}

    Pulling back via \(\sigma_k^{-1}\), the tensor relations (a)--(c) map
    directly to the geometric relations (1)--(3). Relation (d) holds
    inherently for general abelian cycles, while relation (e), combined with
    (d), translates precisely to relation (4) of the corollary, completing the
    proof.
\end{proof}

\begin{remark}
    While this paper was in preparation, Gaifullin~\cite{gaifullin26} proved
    that the homology group \(H_k(\I_g)\) is finitely generated for \(k \leq
    g-2\). We note, however, that the results of~\cite{gaifullin26} do not
    yield an explicit linear-algebraic description or a reasonable upper bound
    on the dimension of \(H_k^{\mathrm{ab,sep}}(\I_g)\). Furthermore, our
    proof of Theorem~\ref{mainthm:fin-dim} is more direct and less involved than
    the proof of finite generation presented in~\cite{gaifullin26}.
\end{remark}

\subsection{Notation and conventions}

By \(g(S)\) and \(b(S)\) we denote
the \emph{genus} and the \emph{number of boundary components} of the surface
\(S\), respectively.

For \(h \in G\), we denote by \([h]\) its homology class in \(H_1(G)\). By
\(G^{(2)}\) we denote the subgroup of the group \(G\) generated by the squares
of the elements of \(G\). 

For homology classes \(a, b \in
H_1(\Sigma_g^b)\), we denote by \(a \cdot b\) their algebraic intersection
number. For integral homology classes \(x, y, \dots\) in \(H_1(\Sigma_g^b)\),
we denote their reductions modulo \(2\) by bold letters \(\bx, \by, \dots\).
Similarly, for symplectic submodules \(V, U, \dots\) of \(H_1(\Sigma_g^b)\),
their reductions modulo \(2\) are denoted by \(\bV, \bU, \dots\).

\subsection{Outline}

The paper is structured as follows. In Section~\ref{sec:preliminaries}, we
briefly review the construction of the Birman--Craggs--Johnson homomorphisms,
the Johnson homomorphism, and the definition of the Torelli group for surfaces
with multiple boundary components. Section~\ref{sec:BCJ-subsurface} is devoted
to computing the kernel of the Birman--Craggs--Johnson homomorphism restricted
to the subgroup of \(\I_g^b\) generated by elements supported on a given
subsurface \(S \subset \Sigma_g^b\) (Theorem~\ref{thm:sigma-kernel}). As a
corollary, we establish a criterion, in terms of the Birman--Craggs--Johnson
and Johnson homomorphisms, for an element \(f \in \I(S)\) to lie in \([\I(S),
\I(S)]\) (Theorem~\ref{thm:torelli-abel-vanish}).

Next, Section~\ref{sec:BCJ-k} derives certain relations between abelian cycles
in \(H_k(\I_g)\) that play a key role in the proofs of
Theorems~\ref{mainthm:fin-dim} and~\ref{mainthm:BCJ-inj}. In
Section~\ref{sec:sigma-k-ab-cyc}, we state and prove a general criterion for
the equality of two abelian cycles (Proposition~\ref{prop:inj-k-2}), which
immediately yields the proof of Theorem~\ref{mainthm:fin-dim}. The proof of a
key technical lemma (Lemma~\ref{lem:p-set-1-subspace}) required for
Proposition~\ref{prop:inj-k-2} is deferred to
Section~\ref{sec:proof-of-auxiliary-lemma}. Section~\ref{sec:symp-submod-BCJ}
then collects several auxiliary properties regarding the images of symplectic
submodules under the Birman--Craggs--Johnson homomorphism. Finally,
Section~\ref{sec:sigma2-inj} is dedicated to the proof of
Theorem~\ref{mainthm:BCJ-inj}.

\subsection{Acknowledgments}

The author is deeply grateful to his advisor, A. A. Gaifullin, for
proposition the research topic and for his continuous guidance and insightful
feedback, which significantly improved the manuscript.

The work was supported by the Theoretical Physics and Mathematics Advancement
Foundation ``BASIS'' (grant 25-8-2-20-1).

    \section{Preliminaries on the Birman--Craggs--Johnson and Johnson
homomorphisms} \label{sec:preliminaries}

\subsection{Birman--Craggs--Johnson homomorphism} 
In this subsection, we briefly recall the construction of the
Birman--Craggs--Johnson homomorphisms; for more details,
see~\cite{johnson-bcj}.

Let \(b \in \{0, 1\}\). Let \(\B(\Sigma_g^b)\) be a commutative algebra over \(\Z/2\Z\),
generated by \(\overline{\bx}\) for all \(\bx \in H_1(\Sigma_g^b;
\Z/2\Z)\) and subject to the relations:
\begin{enumerate}
    \item \(\overline{\bx+\by} = \overline{\bx} + \overline{\by} + (\bx
    \cdot \by)\), where \(\bx \cdot \by\) denotes the intersection number
    \(\bmod\; 2\);
    \item \(\overline{\bx}^2 = \overline{\bx}\).
\end{enumerate}

Thus, for a basis \(\be_1, \dots, \be_{2g}\) of \(H_1(\Sigma_g^b;
\Z/2\Z)\) we get that \(\B(\Sigma_g^b)\) is the algebra of Boolean
polynomials in formal variables \(\overline{\be}_1, \dots,
\overline{\be}_{2g}\).

The \emph{Arf invariant} is the quadratic polynomial given by
\[
\Arf = \sum_{j=1}^{g} \overline{\ba}_j \overline{\bb}_j
,\]
where \(\{\ba_1, \bb_1, \dots, \ba_g, \bb_g\}\) is a symplectic basis for
\(H_1(\Sigma_g^b; \Z/2\Z)\). It is well-known that \(\Arf\) does not depend on
the choice of the symplectic basis.

Let \(\B'(\Sigma_g^b) = \B(\Sigma_g^b) / \left( \Arf \right)\), and let
\(\B_k'(\Sigma_g^b) \subset \B'(\Sigma_g^b)\) denote the space of polynomials of
degree at most \(k\). We will simply write \(\B_k\) and \(\B_k'\) when the
surface is clear from the context.

\emph{The Birman--Craggs--Johnson homomorphisms} are
\begin{align*}
    \sigma &\colon \I_g^1 \to \B_3, \\
    \sigma &\colon \I_g \to \B_3'
.\end{align*}
We denote both maps by the same symbol \(\sigma\), as the context or the
domain will prevent any confusion.

The Birman--Craggs--Johnson homomorphism has a deep topological origin,
arising from Rokhlin invariant of \(3\)-dimensional homology spheres. We will
need the following formula for the value of the Birman--Craggs--Johnson
homomorphism \(\sigma \colon \I_g \to \B_3'\) on the Dehn twist about a
separating simple closed curve~\cite[Lemma~12a]{johnson-bcj}:
\begin{equation} \label{eq:BCJ-sep}
    \sigma(T_{\gamma}) = \sum_{i=1}^{g'} \overline{\ba}_i \overline{\bb}_i
,\end{equation}
where \(g'\) is the genus of a subsurface \(R\) bounded by \(\gamma\), and
\(\{\ba_1, \bb_1, \dots, \ba_{g'}, \bb_{g'}\}\) is a symplectic basis for
\(H_1(R; \Z/2\Z)\). Note that since \(\Arf = 0\) in \(\B_3'\), this
expression is well-defined; that is, it is independent of the choice of the
subsurface \(R\) bounded by \(\gamma\).

We call a subgroup \(V \subset H_1(\Sigma_g)\) a \emph{symplectic submodule}
if the restriction of the intersection form is unimodular. Equivalently, this
means that there is a splitting \(H_1(\Sigma_g) = V \oplus V^{\perp}\) with
respect to the intersection form. A \emph{symplectic subspace} \(\bV \subset
H_1(\Sigma_g; \Z/2\Z)\) is defined analogously.

Let \(\gamma\) be a separating simple closed curve \(\gamma\) on \(\Sigma_g\),
inducing splitting \(H_1(\Sigma_g) = V \oplus V^{\perp}\).
Formula~\eqref{eq:BCJ-sep} implies that the value \(\sigma(T_\gamma)\) depends only on the
unordered splitting \(\{V, V^{\perp}\}\).
We define
\[
\sigma(V) = \sigma(V^{\perp}) = \sigma(T_\gamma)
.\]
Moreover, \(\sigma(V)\) depends only on \(\bV\) (the reduction of \(V\) modulo
\(2\)). We then set \(\sigma(\bV) = \sigma(V)\).

The following proposition is direct corollary of~\cite[Lemma~7.7]{vladimirov}.
\begin{proposition} \label{prop:mod-2-sigma}
    Let \(\bU, \bV\) be \(2\)-dimensional symplectic subspaces of
    \(H_1(\Sigma_g; \Z/2\Z)\). Then we have \(\sigma(\bU) = \sigma(\bV)\) if and only if
    \(\bU = \bV\).
\end{proposition}

\subsection{Torelli groups with multiple boundary components}

Recall the definition of the Torelli group for surfaces with multiple
boundary components (for details, see~\cite{putman-johnsonkernel, putman-cutting}).

Let \(S\) be a connected compact surface. To define the Torelli group
of the surface \(S\), we consider an embedding \(i \colon S \hookrightarrow
\Sigma_g^b\), where \(b \in \{0,1\}\). This embedding induces a homomorphism
\(i_{*} \colon \Mod(S) \to \Mod(\Sigma_g^b)\). We then define
\[
\I(S, \Sigma_g^b) = i_{*}^{-1}(\I_g^b)
.\]
Putman~\cite{putman-cutting} showed that the group \(\I(S, \Sigma_g^b)\)
depends only on the induced partition \(P\) of the boundary components
of \(S\).

\subsection{Johnson homomorphism}

In the 1980s, Johnson~\cite{johnson-ab} constructed the homomorphisms from
Torelli groups to the free abelian groups
\begin{align*}
    \tau &\colon \I_g^1 \to \wedge^3 H_1(\Sigma_g^1), \\
    \tau &\colon \I_g \to \wedge^3 H_1(\Sigma_g) / H_1(\Sigma_g).
\end{align*}
Here, the inclusion \(H_1(\Sigma_g) \hookrightarrow \wedge^3 H_1(\Sigma_g)\)
is given by \(x \mapsto \Omega \wedge x\), where \(\Omega \in \wedge^2
H_1(\Sigma_g)\) is the element dual to the intersection form. We denote both
homomorphisms by the same symbol; this does not lead to any confusion.

Johnson~\cite[Lemma~4A]{johnson-ab} showed that \(\tau(T_{\gamma}) = 0\) for
any separating simple closed curve \(\gamma\) on \(\Sigma_g\). Later,
Johnson~\cite{johnson2} proved that \(\ker \tau = \K_g\), where
\(\K_g \subset \I_g\) denotes the subgroup generated by all Dehn
twists about separating simple closed curves, now called the \emph{Johnson
kernel}. Putman~\cite{putman-johnsonkernel} further extended Johnson's results
to Torelli groups with multiple boundary components.

    \section{The kernel of the Birman--Craggs--Johnson homomorphism for
subsurfaces} \label{sec:BCJ-subsurface}

% Inclusion \(\Sigma_g^b \hookrightarrow \Sigma_{g+b}\) induces inclusion
% \(\I_g^b \hookrightarrow \I_{g+b}\). So we can consider \(S \subset
% \Sigma_{g+b}\) and consider \(\I(S, \Sigma_{g+b})\).

% restriction of johnson homomorphism from \(\Sigma_{g+b}\) to \(S\) and
% restriction of BCJ homomorphism from \(\Sigma_{g+b}\) to \(S\)

% Is it something like 

% \(\wedge^3 H_1(\overline{S}) / H_1(\overline{S}) \oplus \B_2'(\overline{S})\)

% image of 

% ---------------------------------------------------------------------

Johnson showed~\cite{johnson3} that for \(b \in \{0,1\}\) and \(g
\geq 3\), the kernel of the Birman--Craggs--Johnson homomorphism \(\C_g^b = \ker \sigma\)
is equal to \((\I_g^b)^{(2)}\). The purpose of this section is to extend this
result to Torelli groups with multiple boundary components.

Throughout this section, we follow the conventions from \cite{putman-cutting} and
\cite{putman-johnsonkernel}. The proof of Theorem~\ref{thm:sigma-kernel}
is heavily inspired by the proof of Theorem~A from~\cite{putman-johnsonkernel}.

\subsection{Capped embeddings}
Let \(\Sigma_g^b\) be a compact surface of genus \(g\) with \(b \leq
1\) boundary components. We call an embedding \(i \colon S \hookrightarrow \Sigma_g\)
\emph{capped} if the connected components of \(\Sigma_g
\setminus S\) are homeomorphic to \(\Sigma_{h_j}^1\) with \(h_j > 0\). Similarly,
an embedding \(i \colon S \hookrightarrow \Sigma_g^1\) is called
\emph{capped} if the connected components of \(\Sigma_g^1
\setminus S\) are homeomorphic to \(\Sigma_{h_j}^1\) with \(h_j > 0\) and an annulus \(A\), where
\(\partial \Sigma_g^1 \subset \partial A\).

Note that, for a capped embedding \(i \colon S \hookrightarrow
\Sigma_g^b\), the induced homomorphism \(i_{*} \colon \Mod(S) \to \Mod(\Sigma_g^b)\)
is injective (see~\cite[Theorem~3.18]{farbmarg}). Thus, for
a capped embedding, the homomorphism \(i_{*}\) induces an isomorphism between the group \(\I(S,
\Sigma_g^b)\) and the subgroup of \(\I_g^b\)
consisting of the elements that are realized by homeomorphisms
supported on \(S\).

In what follows, we consider only capped embeddings; that is, by \(\I(S,
\Sigma_g^b)\) we imply that the embedding \(S \hookrightarrow
\Sigma_g^b\) is capped.
In addition, we will sometimes write simply \(\I(S)\) instead of \(\I(S, \Sigma_g^b)\),
since for a capped embedding, the partition of the boundary components
\(\partial S\) is always \(P = \{\{\partial_1\}, \dots,
\{\partial_m\}\}\), where \(\partial S = \partial_1 \cup \dots \cup
\partial_m\).

\subsection{Relations between Dehn twists in \(H_1(\I(S))\)}

In this subsection, we derive certain relations between Dehn twists about
separating simple closed curves in \(H_1(\I(S))\). Note that although we know
from~\cite[Corollary to Lemma~3]{johnson3} that \(2 [T_{\gamma}] = 0\) in
\(H_1(\I_g^b)\), it is a priori unknown whether this relation holds in
\(H_1(\I(S))\). We prove that this is indeed the case
(see~Lemma~\ref{lem:abel-ord-2}). This will imply that the order of the
summands in~\eqref{eq:sum-relation-1} and~\eqref{eq:sum-relation-2} does not
matter.

\begin{lemma} \label{lem:sum-relation-1}
    Let \(\gamma_1, \gamma_2, \gamma_3\) be pairwise disjoint, nonisotopic,
    separating simple closed curves on \(S\) such that:
    \begin{itemize}
        \item \(\gamma_1, \gamma_2\) are of genus \(1\);
        \item \(\gamma_1 \cup \gamma_2 \cup \gamma_3\) bounds a pair of pants.
    \end{itemize}
    We then have
    \begin{equation} \label{eq:sum-relation-1}
        [T_{\gamma_1}] + [T_{\gamma_3}] = [T_{\gamma_2}]
    .\end{equation}
\end{lemma}

\begin{proof}
    The proof of the lemma follows step-by-step the proof of the Lemma~5.1
    from~\cite{vladimirov}.
\end{proof}

\begin{lemma} \label{lem:sum-relation-2}
    Let \(\gamma_1, \gamma_2, \gamma_3\) be pairwise disjoint, nonisotopic,
    separating simple closed curves on \(S\) such that
    \begin{itemize}
        \item \(\gamma_1 \cup \gamma_2 \cup \gamma_3\) bounds a pair of pants;
        \item the subsurface bounded by \(\gamma_3\), that
        does not contain \(\gamma_1 \cup \gamma_2\) has positive
        genus.
    \end{itemize}
    Then we have
    \begin{equation} \label{eq:sum-relation-2}
        [T_{\gamma_1}] + [T_{\gamma_2}] = [T_{\gamma_3}]
    .\end{equation}
\end{lemma}

\begin{proof}
    The proof of the lemma follows step-by-step the proof of the Lemma~5.6
    from~\cite{vladimirov} in case \(k = 1\).
\end{proof}

\subsection{Definition of the surfaces \(\widehat{S}\) and \(S'\) for \(S\)}

Consider a capped embedding \(S \hookrightarrow \Sigma_g^b\), where \(b(S)
\geq 2\). Let \(\beta\) denote one of the boundary components of \(\partial
S\); in the case where \(b = 1\), we choose the component of \(\partial S\)
that is isotopic to \(\partial \Sigma_g^1\). We denote by \(\widehat{S}\) the
surface obtained from \(S\) by gluing a disk to the boundary component
\(\beta\). Finally, we choose a compact subsurface \(S' \subset S\) such that
the subsurface \(S \setminus \Int S'\) is homeomorphic to a pair of pants,
where \(\beta\) is a boundary component of this pair of pants
(see~Fig.~\ref{fig:splitting-subsurface}).

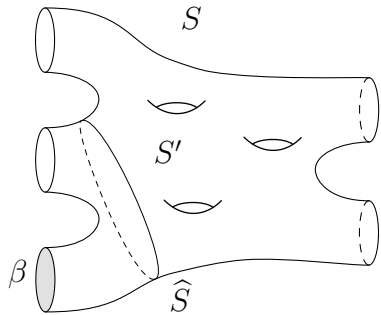
\begin{figure}[H]
    \centering
    \def\svgwidth{.3\columnwidth}%
    %% Creator: Inkscape 1.4.2 (ebf0e940d0, 2025-05-08), www.inkscape.org
%% Accompanies image file 'splitting-subsurface.pdf' (pdf, eps, ps)
\begingroup%
  \makeatletter%
  \providecommand\color[2][]{%
    \errmessage{(Inkscape) Color is used for the text in Inkscape, but the package 'color.sty' is not loaded}%
    \renewcommand\color[2][]{}%
  }%
  \providecommand\transparent[1]{%
    \errmessage{(Inkscape) Transparency is used (non-zero) for the text in Inkscape, but the package 'transparent.sty' is not loaded}%
    \renewcommand\transparent[1]{}%
  }%
  \providecommand\rotatebox[2]{#2}%
  \newcommand*\fsize{\dimexpr\f@size pt\relax}%
  \newcommand*\lineheight[1]{\fontsize{\fsize}{#1\fsize}\selectfont}%
  \ifx\svgwidth\undefined%
    \setlength{\unitlength}{202.6953966bp}%
    \ifx\svgscale\undefined%
      \relax%
    \else%
      \setlength{\unitlength}{\unitlength * \real{\svgscale}}%
    \fi%
  \else%
    \setlength{\unitlength}{\svgwidth}%
  \fi%
  \global\let\svgwidth\undefined%
  \global\let\svgscale\undefined%
  \makeatother%
  \begin{picture}(1,0.85643027)%
    \lineheight{1}%
    \setlength\tabcolsep{0pt}%
    \put(0,0){\includegraphics[width=\unitlength,page=1]{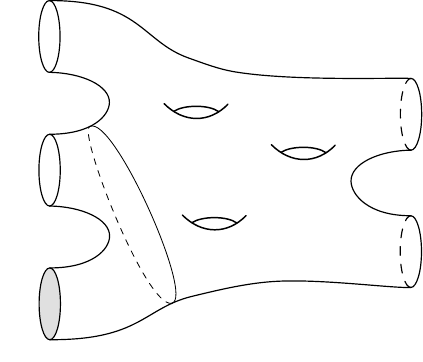}}%
    \put(0.50402688,0.80037375){\color[rgb]{0,0,0}\makebox(0,0)[t]{\lineheight{1.25}\smash{\begin{tabular}[t]{c}$S$\end{tabular}}}}%
    \put(0.04484565,0.13196148){\color[rgb]{0,0,0}\makebox(0,0)[t]{\lineheight{1.25}\smash{\begin{tabular}[t]{c}$\beta$\end{tabular}}}}%
    \put(0.4425125,0.44432582){\color[rgb]{0,0,0}\makebox(0,0)[t]{\lineheight{1.25}\smash{\begin{tabular}[t]{c}$S'$\end{tabular}}}}%
    \put(0.47350202,0.0547295){\color[rgb]{0,0,0}\makebox(0,0)[t]{\lineheight{1.25}\smash{\begin{tabular}[t]{c}$\widehat{S}$\end{tabular}}}}%
  \end{picture}%
\endgroup%

    \caption{Surfaces \(S'\) and \(\widehat{S}\) for \(S\).}
    \label{fig:splitting-subsurface}
\end{figure}

In what follows, \(S'\) and \(\widehat{S}\) will always denote the surfaces
defined above for a capped embedding \(S \hookrightarrow \Sigma_g^b\).

\subsection{Statement of the theorem}
Let \(b \in \{0, 1\}\). Consider a capped embedding \(i \colon S \hookrightarrow
\Sigma_g^b\), and define
\[
\C(S, \Sigma_g^b) = i_{*}^{-1} (\C_g^b)
,\]
where \(\C_g^b = \ker \sigma_{\Sigma_g^b}\) is the kernel of the
Birman--Craggs--Johnson homomorphism.

\begin{theorem} \label{thm:sigma-kernel}
    Let \(S \hookrightarrow \Sigma_g^b\) be a capped embedding, where \(b \in
    \{0, 1\}\). Then for \(g(S) \geq 3\), we have
    \[
    \C(S, \Sigma_g^b) = (\I(S))^{(2)}.
    \]
\end{theorem}

\begin{remark}
    Although the group \(\I(S, \Sigma_g^b)\) does not depend on the choice of
    the capped embedding \(S \hookrightarrow \Sigma_g^b\), its subgroup
    \(\C(S, \Sigma_g^b) \subset \I(S, \Sigma_g^b)\) a priori depends
    on the choice of the embedding.
\end{remark}

\begin{corollary}
    For \(g(S) \geq 3\), the subgroup \(\C(S) = \C(S, \Sigma_g^b)\) does not
    depend on the choice of the capped embedding \(S \hookrightarrow
    \Sigma_g^b\).
\end{corollary}

The proof of Theorem~\ref{thm:sigma-kernel} is heavily inspired
by the proof of~\cite[Theorem~5.1]{putman-johnsonkernel}.
We will need the following auxiliary Proposition~\ref{prop:semidirect}
(cf.~\cite[Theorem~4.4]{putman-johnsonkernel}).

As shown by Putman
(see~\cite[Theorem~2.2]{putman-johnsonkernel}
and~\cite[Theorem~1.2]{putman-cutting}), we have 
the following Birman exact sequence:
\[
\begin{tikzcd}[column sep = 1.5em]
    1 \arrow[r] & \PP(S, \Sigma_g^b; \beta) \arrow[r] &\I(S, \Sigma_g^b)
    \arrow[r] & \I(\widehat{S}) \arrow[r] &1
,\end{tikzcd}
\]
where \(\PP(S, \Sigma_g^b; \beta) \cong \pi_1(U \widehat{S})\). Here \(U
\widehat{S}\) denotes the unit tangent bundle for \(\widehat{S}\). We then
define
\[
\PP_{\C}(S, \Sigma_g^b; \beta) = \PP(S, \Sigma_g^b; \beta) \cap \C(S,
\Sigma_g^b)
.\]

\begin{proposition} \label{prop:semidirect}
    Let \(S \hookrightarrow \Sigma_g^b\) be a capped embedding with \(b \in
    \{0, 1\}\) such that \(b(S) \geq 2\) and \(g(S) \geq 2\). Then
    \[
    \C(S, \Sigma_g^b) = \PP_{\C}(S, \Sigma_g^b; \beta) \rtimes \C(S',
    \Sigma_g^b)
    .\]
\end{proposition}

The proof is based on the following observation.

\begin{lemma} \label{lem:ker-semidirect}
    Consider a group \(G = K \rtimes Q\).
    Then, for a homomorphism \(f \colon G \to H\), we have
    \[
    \ker f = \ker (f|_{K}) \rtimes \ker (f|_{Q})
    \]
    if and only if \(f(K) \cap f(Q) = 1\).
\end{lemma}

We will also need the following elementary linear-algebraic lemma.

\begin{lemma} \label{lem:arf-zero-intersec}
    Let \(S\) be subsurface of \(\Sigma_g^1\) such that 
    \begin{itemize}
        \item \(\Sigma_g^1 \setminus S\) has positive genus; and
        \item \(S\) has one boundary component.
    \end{itemize}
    Then in \(\B_3(\Sigma_g^1)\) we have
    \[
    \B_1 (\Sigma_g^1) \Arf(\Sigma_g^1) \cap \B_3(S) = 0
    .\] 
\end{lemma}

\begin{proof}
    We will show that no non-zero element \(f \in \B_1(\Sigma_g^1)
    \Arf(\Sigma_g^1)\) belongs to \(\B_3(S)\). Let \(f = (c + \overline{\bx})
    \Arf(\Sigma_g^1)\), where \(c \in \{0, 1\}\) and \(\bx \in H_1(\Sigma_g^1;
    \Z/2\Z)\). If \(f \in \B_3(S)\), we must show that \(f = 0\).

    Consider the splitting \(H_1(\Sigma_g^1; \Z/2\Z) = H_1(S; \Z/2\Z) \oplus
    H_1(S; \Z/2\Z)^{\perp}\), and let \(\bx = \bx_0 + \bx_1\) be the
    corresponding decomposition of \(\bx\). Let \(\{\ba_1,
    \bb_1, \dots, \ba_g, \bb_g\}\) be a symplectic basis for \(H_1(\Sigma_g^1; \Z/2\Z)\) such that:
    \begin{itemize}
        \item \(\{\ba_j, \bb_j, \dots, \ba_g, \bb_g\}\) is a symplectic
        basis for \(H_1(S; \Z/2\Z)\) (for some \(j \geq 2\));
        \item \(\bx_0 \in \{0, \ba_1\}\) and \(\bx_1 \in \{0, \ba_j\}\).
    \end{itemize}

    Assume that \(f \in \B_3(S)\), which implies that
    \[
    c \Arf(S^{\perp}) + \epsilon_0 \overline{\ba}_1 \Arf(\Sigma_g^1) + \epsilon_1 \overline{\ba}_j \Arf(\Sigma_g^1) \in \B_3(S)
    ,\]
    where \(\epsilon_0, \epsilon_1 \in \{0,1\}\) and
    \[
    \Arf(S^{\perp}) = \Arf(\Sigma_g^1) - \Arf(S) = \sum_{i=1}^{j-1} \overline{\ba}_i \overline{\bb}_i
    .\]
    If \(j > 2\), then linear independence requires \(c = 0\), and therefore
    \(\epsilon_0 = 0\). However, \(\epsilon_1 \overline{\ba}_j
    \Arf(\Sigma_g^1) \in \B_3(S)\) is impossible for \(\epsilon_1 \neq 0\),
    which forces \(f = 0\).

    For \(j = 2\), the terms containing \(\overline{\ba}_1\) and \(\overline{\bb}_1\) are
    \[
    (c + \epsilon_0) \overline{\ba}_1 \overline{\bb}_1 + \epsilon_1 \overline{\ba}_j \overline{\ba}_1 \overline{\bb}_1
    .\]
    This forces \(c = \epsilon_0\) and \(\epsilon_1 = 0\). If \(c = \epsilon_0
    = 0\), then \(f = 0\), a contradiction. If instead \(c = \epsilon_0 = 1\),
    we would then have \((1 + \overline{\ba}_1) \Arf(S^{\perp}) +
    \overline{\ba}_1 \Arf(S) \in \B_3(S)\), meaning \(\overline{\ba}_1 \Arf(S)
    \in \B_3(S)\), which is impossible since \(\ba_1 \notin H_1(S; \Z/2\Z)\).
    Therefore, if \(f \in \B_3(S)\), then \(f = 0\), and the lemma follows.
\end{proof}

\begin{proof}[Proof of Proposition~\ref{prop:semidirect}]
    By a theorem of Putman~\cite[Theorem~1.2]{putman-cutting} (see also the
    reformulation in \cite[Theorem~2.3]{putman-johnsonkernel}), we have the
    decomposition:
    \[
    \I(S, \Sigma_g^b) = \PP(S, \Sigma_g^b; \beta) \rtimes \I(S',
    \Sigma_g^b)
    .\]
    By Lemma~\ref{lem:ker-semidirect}, to prove the proposition,
    it suffices to show that
    \[
    \sigma(\PP(S, \Sigma_g^b; \beta)) \cap \sigma(\I(S', \Sigma_g^b)) = 0
    .\]
    Johnson showed~\cite[Lemma~16]{johnson3} that \(\sigma(\PP(\Sigma_g^1;
    \beta)) = \B_1(\Sigma_g^1) \Arf(\Sigma_g^1)\). Thus,
    \[
    \sigma(\PP(S, \Sigma_g^b; \beta)) = \sigma(\PP(S; \beta)) \subset
    \sigma(\PP(\Sigma_h^1; \beta)) = \B_1(\Sigma_{h}^1) \Arf(\Sigma_h^1)
    ,\]
    where \(\Sigma_h^1\) is the subsurface of \(\Sigma_g^b\) defined as follows:
    \begin{itemize}
        \item if \(b = 1\), then  \(\Sigma_h^1 = \Sigma_g^1\);
        \item if \(b = 0\), then \(\Sigma_h^1\) is a subsurface of
        \(\Sigma_g\) bounded by \(\beta\) and containing \(S\).
    \end{itemize}

    Consider the subsurface \(S''\) in \(\Sigma_g^b\) defined as follows. Let
    \(\beta, \beta', \beta''\) denote the boundary components of \(S \setminus \Int S'\),
    where \(\beta'' \subset \partial S'\). We denote by \(S''\) the
    subsurface in \(\Sigma_g^b\) bounded by the curve \(\beta''\) that does not
    contain \(\beta\) and \(\beta'\)
    (see~Fig.~\ref{fig:s-double-prime-surface}).

    \begin{figure}[H]
        \centering
    \def\svgwidth{.5\columnwidth}%
    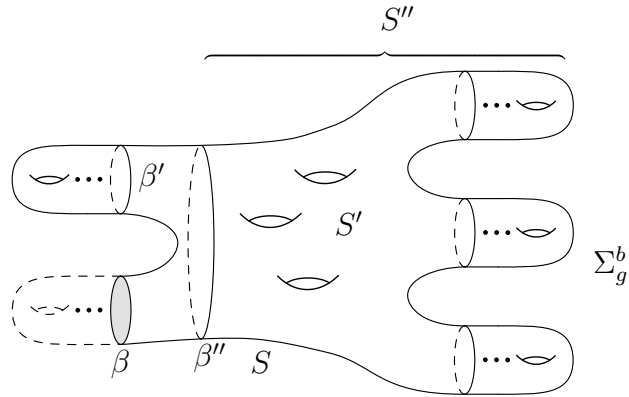%

        \caption{Subsurface \(S''\) in \(\Sigma_g^b\).}
        \label{fig:s-double-prime-surface}
    \end{figure}

    We have \(\I(S', \Sigma_g^b) = \I(S', S'')\), which means that
    \[
    \sigma(\I(S', \Sigma_g^b)) = \sigma(\I(S', S'')) \subset \B_3(S'')
    .\]

    Thus, it suffices to show that in \(\B_3(\Sigma_h^1)\) we have
    \[
    \B_1 (\Sigma_h^1) \Arf(\Sigma_h^1) \cap \B_3(S'') = 0
    .\]
    The proposition now follows from Lemma~\ref{lem:arf-zero-intersec} applied
    to \(S'' \subset \Sigma_h^1\).
\end{proof}

\begin{lemma} \label{lem:disk-pushing}
    Let \(S \hookrightarrow \Sigma_g^b\) be a capped embedding, where \(b \in
    \{0, 1\}\). Let \(\beta\) denote the boundary component of \(\partial S\)
    that is isotopic to \(\partial \Sigma_g^1\) when \(b = 1\). Then for
    \(g(S) \geq 3\) and \(b(S) \geq 2\), we have
    \[
    \PP_{\C}(S, \Sigma_g^b; \beta) \subset (\I(S))^{(2)}
    .\]
\end{lemma}

Before proving Lemma~\ref{lem:disk-pushing}, let us consider the following
case. Consider the embedding \(\Sigma_g^1 \hookrightarrow \Sigma_g\) obtained
by gluing a two-dimensional disk to the boundary component \(\beta = \partial \Sigma_g^1\). We then have the Birman exact
sequence:
\[
\begin{tikzcd}[column sep = 1.5em]
    1 \arrow[r] & \PP(\Sigma_g^1; \beta) \arrow[r] &\I_g^1 \arrow[r] &\I_g
    \arrow[r] &1
.\end{tikzcd}
\]
Here, \(\PP(\Sigma_g^1; \beta) \cong \pi_1(U \Sigma_g)\).
Johnson showed~\cite[Lemma~16]{johnson3} that \(\sigma (
\PP(\Sigma_g^1; \beta) ) = \B_1 \Arf(\Sigma_g^1)\) for the
Birman--Craggs--Johnson homomorphism \(\sigma \colon \I_g^{1} \to \B_3\).
In addition, we
have the following proposition.

\begin{proposition} \label{prop:johnson-disk-pushing}
The following diagram commutes:
\[
\begin{tikzcd}[column sep = 2em]
    \PP(\Sigma_g^1; \beta) \arrow[d, "\cong"'] \arrow[rr, "\sigma",
    twoheadrightarrow] & &[1.5em] \B_1 \Arf(\Sigma_g^1) \\
    \pi_1(U \Sigma_g) \arrow[r, twoheadrightarrow] &H_1(U \Sigma_g) \arrow[r,
    twoheadrightarrow, "\mathrm{mod}\, 2"] & H_1(\Sigma_g; \Z/2\Z)
    \oplus (\Z/2\Z) \arrow[u, "\cong"']
\end{tikzcd}
\]
Here, the homomorphism \(H_1(\Sigma_g; \Z/2\Z) \oplus (\Z/2\Z) \to \B_1
\Arf(\Sigma_g^1)\) is given by
\begin{align*}
    \ba_i &\mapsto (\overline{\ba}_i + 1) \Arf(\Sigma_g^1) \\
    \bb_i &\mapsto (\overline{\bb}_i + 1) \Arf(\Sigma_g^1) \\
    \bz &\mapsto \Arf(\Sigma_g^1)
,\end{align*}
where \(\{\ba_1, \bb_1, \dots, \ba_g, \bb_g\}\) is a symplectic basis for
\(H_1(\Sigma_g; \Z/2\Z)\), and \(\bz\) is the homology class of the fiber of
the unit tangent bundle \(U \Sigma_g \to \Sigma_g\).
\end{proposition}

\begin{proof}[Proof of Lemma~\ref{lem:disk-pushing}]
    Putman~\cite[Theorem~1.2]{putman-cutting} (see also the reformulation
    in~\cite[Theorem~2.2]{putman-johnsonkernel}) proved that
    \[
    \PP(S, \Sigma_g^b; \beta) = \PP(S; \beta) \cong \pi_1(U \widehat{S})
    .\]

    We first show that it suffices to prove the lemma in the case \(b = 1\),
    that is,
    \[
    \PP_{\C}(S, \Sigma_{g'}; \beta) = \PP_{\C}(S, \Sigma_{g}^1; \beta)
    ,\]
    for an embedding \(\Sigma_{g}^1 \hookrightarrow \Sigma_{g'}\) with \(g < g'\). Indeed, for any \(f \in
    \I(S)\), we have \(\sigma_{\Sigma_{g}^1}(f) = 0\) if and only if
    \(\sigma_{\Sigma_{g'}}(f) = 0\). It follows that
    \[
    \PP_{\C}(S, \Sigma_{g'}; \beta) = \PP(S, \Sigma_{g'}; \beta) \cap \ker
    (\sigma_{\Sigma_{g'}})|_{S} = \PP(S, \Sigma_{g}^1; \beta) \cap \ker
    (\sigma_{\Sigma_{g}^1})|_{S} = \PP_{\C}(S, \Sigma_{g}^1; \beta)
    .\]

    Thus, it suffices to prove the claim for \(\PP_{\C}(S, \Sigma_g^1; \beta)\).
    Let \(\Sigma_g = \overline{\Sigma_g^1}\) denote the surface obtained by
    gluing a disk to the boundary component \(\beta = \partial \Sigma_g^1\).

    We then have the following commutative diagram:
    \[
    \begin{tikzcd}
        \PP(\Sigma_g^1; \beta) \arrow[r, "\cong"] &\pi_1(U \Sigma_g) \arrow[r, twoheadrightarrow] &H_1(U \Sigma_g)
        \arrow[r, twoheadrightarrow, "\bmod 2"] &[1em] (\Z/2\Z)^{2g}
        \oplus (\Z/2\Z) \\
        \PP(S, \Sigma_g^1; \beta) \arrow[u, hook'] \arrow[r, "\cong"] &\pi_1(U
        \widehat{S}) \arrow[u, "\iota_*", hook'] \arrow[r, twoheadrightarrow] &H_1(U
        \widehat{S}) \arrow[u, "\iota_*"] 
    \end{tikzcd}
    \]
    By Proposition~\ref{prop:johnson-disk-pushing}, the top row
    coincides with \(\sigma \colon \PP(\Sigma_g^1;
    \beta) \to \B_1 \Arf(\Sigma_g^1)\). Thus, we see that
    \(\PP_{\C}(S, \Sigma_g^1; \beta)\) is precisely the kernel of the following composition:
    \[
    \begin{tikzcd}[column sep = 2.2em]
        \PP(S, \Sigma_g^1; \beta) \rar["\cong"] &\pi_1(U \widehat{S})
        \rar[twoheadrightarrow, "\phi"] & H_1(U \widehat{S}) \rar["i_*"] &H_1( U \Sigma_g)
        \rar[twoheadrightarrow, "\bmod 2"] &[1em] (\Z/2\Z)^{2g} \oplus
        (\Z/2\Z)
    .\end{tikzcd}
    \]

    Choose a basepoint \(\ast \in \beta\) and consider the following simple closed curves:
    \begin{itemize}
        \item \(\gamma_1, \dots, \gamma_n \in \pi_1(\widehat{S}, \ast)\),
        which are isotopic to the boundary components of \(\widehat{S}\);
        \item \(\zeta_1, \dots, \zeta_{2h} \in \pi_1(\widehat{S}, \ast)\) such
        that \([\zeta_1] = a_1, [\zeta_2] = b_1, \dots, [\zeta_{2h-1}] =
        a_h, [\zeta_{2h}] = b_h\).
    \end{itemize}
    We then have
    \begin{align*}
        (i_{*} \circ \phi)^{-1} (\langle 2 a_1, 2 b_1, \dots, 2 a_g, 2 b_g, 2 z \rangle)
        &= \phi^{-1} \left( \langle 2 a_1, 2 b_1, \dots, 2 a_h, 2 b_h, 2 z, c_1 + z, \dots, c_n + z \rangle \right) \\
        &= \langle [\pi_1(U \widehat{S}), \pi_1(U \widehat{S})],
        \overline{\zeta}_1^2, \dots, \overline{\zeta}_{2h}^2, z^2,
        \overline{\gamma}_1 z, \dots, \overline{\gamma}_n z \rangle
    ,\end{align*}
    where \(c_j = [\gamma_j] \in H_1(\widehat{S})\), \(\overline{\gamma}_j \in
    \pi_1(U \widehat{S})\) denotes the canonical tangent lift of \(\gamma_j \in
    \pi_1(\widehat{S})\), and \(z \in \pi_1(U
    \widehat{S})\) corresponds to the generator of the fiber for the bundle
    \(U \Sigma_g \to \Sigma_g\).

    \begin{figure}[H]
        \centering
    \def\svgwidth{.35\columnwidth}%
    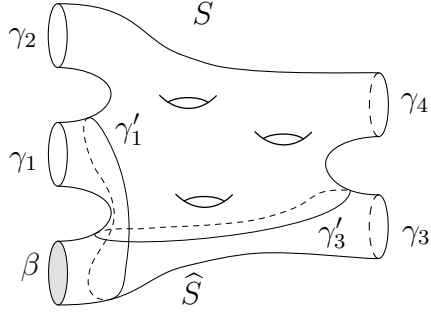%

        \caption{Example of the curves \(\gamma_1'\) and \(\gamma_3'\) on \(S\).}
        \label{fig:bp-set}
    \end{figure}

    Since \([\pi_1(U \widehat{S}), \pi_1(U \widehat{S})] \subset (\pi_1(U
    \widehat{S}))^{(2)}\), we have the
    following inclusion:
    \[
    \ker \left( \mathrm{mod}\,2 \circ i_{*} \circ \phi \right) \subset \langle
    (\pi_1(U \widehat{S}))^{(2)}, \overline{\gamma}_1 z, \dots,
    \overline{\gamma}_n z \rangle
    .\]

    Under the isomorphism \(\pi_1(U \widehat{S}) \cong \PP(S,
    \Sigma_g^1; \beta)\), the element \(\overline{\gamma}_i z\) maps to
    \(T_{\gamma_i} T_{\gamma_i'}^{-1} T_{\beta}\) for all \(i = 1, \dots,
    n\). By Lemma~\ref{lem:sum-relation-2}, in \(H_1(\I(S))\) we have
    \[
    [T_{\gamma_i} T_{\gamma_i'}^{-1} T_{\beta}] = [T_{\gamma_i}] + [T_{\beta}] - [T_{\gamma_i'}] = 0
    .\]
    Thus, \(T_{\gamma_i} T_{\gamma_i'}^{-1} T_{\beta} \in [\I(S), \I(S)]
    \subset \I(S)^{(2)}\), and the lemma follows.
\end{proof}

\begin{proof}[Proof of Theorem~\ref{thm:sigma-kernel}]
    Note that for any capped embedding \(S \hookrightarrow
    \Sigma_g^b\), we have the inclusion \((\I(S))^{(2)} \subset \C(S, \Sigma_g^b)\).
    Thus, it remains to prove the reverse inclusion. We proceed by induction
    on the number of boundary components \(b(S)\).

    \emph{Base case.} If \(b(S) = 1\), the statement is precisely a theorem of
    Johnson \cite[Theorem~2]{johnson3}, which states that \(\C(S, \Sigma_g^b)
    = \ker \sigma_S = (\I(S))^{(2)}\), where \(\sigma_S \colon \I(S) \to
    \B_3(S)\).

    \emph{Inductive step.} Suppose the statement holds for all \(S\) with
    \(b(S) \leq k\). Let us prove that it also holds for \(b(S) = k + 1\).

    By Proposition~\ref{prop:semidirect}, we have 
    \[
    \C(S, \Sigma_g^b) = \PP_{\C}(S, \Sigma_g^b; \beta) \rtimes \C(S',
    \Sigma_g^b)
    .\]
    By the induction hypothesis, \(\C(S', \Sigma_g^b) = (\I(S'))^{(2)}\).
    Since \(S' \subset S\), the inclusion \((\I(S'))^{(2)} \subset
    (\I(S))^{(2)}\) holds, which implies that \(\C(S', \Sigma_g^b) \subset
    (\I(S))^{(2)}\). Furthermore, Lemma~\ref{lem:disk-pushing} yields
    \(\PP_{\C}(S, \Sigma_g^b; \beta) \subset (\I(S))^{(2)}\). From the
    semidirect product decomposition, it follows that
    \(\C(S, \Sigma_g^b) \subset (\I(S))^{(2)}\), completing the proof.
\end{proof}

\subsection{Abelianization of \(\I(S)\)}

The main result of this subsection is the following theorem.

\begin{theorem} \label{thm:torelli-abel-vanish}
    An element \(f \in \I(S)\) lies in \([\I(S), \I(S)]\) if and only if for
    some capped embedding
    \(S \hookrightarrow \Sigma_g^b\) (with \(b \in \{0, 1\}\)) we have
    \(\sigma_{\Sigma_g^b}(f) = 0\) and \(\tau_{\Sigma_g^b}(f) = 0\).
\end{theorem}

This is a direct generalization of a result of Johnson~\cite[Theorem~3(b),
Theorem~6(b)]{johnson3}. We begin with the following lemma.

\begin{lemma} \label{lem:abel-ord-2}
    Let \(\gamma\) be a separating simple closed curve on \(S\) with \(g(S)
    \geq 3\). Then \(2 [T_{\gamma}] = 0\) in \(H_1(\I(S))\).
\end{lemma}

\begin{proof}
    If \(\gamma\) is not a boundary component, then
    Lemma~\ref{lem:sum-relation-1} implies that it suffices to prove the lemma
    in the case where the genus of \(\gamma\) is \(1\). Indeed, suppose that
    \(\gamma\) bounds a subsurface \(R\). Then we may choose disjoint
    separating simple closed curves \(\gamma_1, \gamma_2\) on \(R\) such
    that:
    \begin{itemize}[leftmargin=1cm]
        \item \(\gamma \cup \gamma_1 \cup \gamma_2\) bounds a pair of pants, and
        \item \(\gamma_1\) has genus~\(1\).
    \end{itemize}
    By Lemma~\ref{lem:sum-relation-1}, we have
    \[
    [T_{\gamma}] + [T_{\gamma_2}] = [T_{\gamma_1}]
    .\]
    Applying the same argument inductively to \(T_{\gamma_2}\) and the
    resulting terms, we find that \([T_{\gamma}]\) can be expressed as a sum
    \[
    [T_{\gamma}] = \sum_{j=1}^{m} [T_{\gamma_j}]
    ,\]
    where each \(\gamma_j\) has genus~\(1\).

    Therefore, in further arguments we assume that \(\gamma\) is either a
    boundary component of \(S\) or a genus-\(1\) separating simple closed
    curve.

    Choose separating simple closed curves \(\delta_1, \delta_2\) such that:
    \begin{itemize}
        \item \(\delta_1\) is of genus \(1\);
        \item \(\gamma \cup \delta_1 \cup \delta_2\) bounds a pair of pants.
    \end{itemize}
    Then by Lemma~\ref{lem:sum-relation-1} we have
    \[
    [T_{\gamma}] + [T_{\delta_2}] = [T_{\delta_1}]
    .\] 
    Since \(g(S) \geq 3\), the subsurface bounded by \(\delta_2\) that does
    not contain \(\gamma \cup \delta_1\) has positive genus. Then
    Lemma~\ref{lem:sum-relation-2} implies that
    \[
    [T_{\gamma}] + [T_{\delta_1}] = [T_{\delta_2}]
    .\] 
    Summing these two relations yields \(2 [T_{\gamma}] = 0\) in \(H_1(\I(S))\).
\end{proof}

\begin{proof}[Proof of Theorem~\ref{thm:torelli-abel-vanish}]
    One direction is vacuous, since both \(\sigma_{\Sigma_g^b}\) and
    \(\tau_{\Sigma_g^b}\) are homomorphisms into abelian groups. So suppose
    that we have \(\sigma_{\Sigma_g^b}(f) = 0\) and \(\tau_{\Sigma_g^b}(f) =
    0\) for some capped embedding \(S \hookrightarrow \Sigma_g^b\). By
    Theorem~\ref{thm:sigma-kernel}
    we know that since \(f \in \C(S)\), we have 
    \[
    f = f_1^2 \cdots f_N^2
    ,\] 
    for some \(f_1, \dots, f_N \in \I(S)\). This implies that in
    \(H_1(\I(S))\) we have 
    \[
    [f] = 2 [f_1] + \dots + 2 [f_N] 
    .\] 
    However we also have 
    \[
    0 = \tau_{\Sigma_g^b}(f) = 2 \tau_{\Sigma_g^b}(f_1) + \dots + 2
    \tau_{\Sigma_g^b}(f_N) = 2 \tau_{\Sigma_g^b}(f_1 \cdots f_N)
    .\] 
    Putman's theorem~\cite[Theorem~A]{putman-johnsonkernel} implies that \(f_1
    \cdots f_N \in \K(S)\), i.e.,
    \[
    f_1 \cdots f_N = T_{\eta_1} \cdots T_{\eta_M}
    ,\] 
    where \(\eta_1, \dots, \eta_M\) are separating simple closed curves on
    \(S\).

    It follows that 
    \[
    [f] = 2 [T_{\eta_1}] + \dots + 2 [T_{\eta_M}]
    .\] 
    By Lemma~\ref{lem:abel-ord-2}, we have \(2 [T_{\eta_j}] = 0\) in
    \(H_1(\I(S))\) for all \(j = 1, \dots, M\). This implies that \([f] = 0\)
    in \(H_1(\I(S))\), as desired.
\end{proof}

\begin{theorem} 
    For any capped embedding \(S \hookrightarrow \Sigma_g^b\) (with \(b \in
    \{0, 1\}\)), the pushforward map \(H_1(\I(S)) \to H_1(\I_g^b)\) is
    injective for \(g(S) \geq 3\).
\end{theorem}

\begin{proof}
    By~\cite[Theorem~3(b), Theorem~6(b)]{johnson3}, for \(f \in \I_g^b\), supported on \(S\), we have
    \([f] \in H_1(\I_g^b)\) vanishes if and only if
    \(\tau_{\Sigma_g^b}(f) = 0\) and \(\sigma_{\Sigma_g^b}(f) = 0\). Then
    Theorem~\ref{thm:torelli-abel-vanish} implies \([f] = 0\) in
    \(H_1(\I(S))\), and the theorem follows.
\end{proof}

    \section{The Birman--Craggs--Johnson homomorphism and relations in \(H_k(\I_g)\)}
\label{sec:BCJ-k}

\subsection{The key relation between abelian cycles}

In this subsection, we introduce and establish a specific relation between
abelian cycles arising from geometric considerations, which serves as an
analogue to~\cite[Proposition~6.1]{vladimirov}.

Recall from~\cite[Section~8]{vladimirov} that a collection of pairwise
disjoint, pairwise nonisotopic separating simple closed curves \(\delta_1,
\dots, \delta_k\) on \(\Sigma_g\) is called an \emph{admissible partition} of
\(\Sigma_g\). Given such a partition, a curve \(\delta_i\) is said to be
\emph{outermost} if it bounds a subsurface of \(\Sigma_g\) containing no other
curves of the collection. This subsurface is called a \emph{cap} over
\(\delta_i\) and is denoted by \(\Cap(\delta_i)\).

\begin{proposition}\label{prop:key-relation-k}
    Let \(\delta_1, \dots, \delta_{k-1}\) be an admissible partition of
    \(\Sigma_g\) such that all curves are outermost, and let
    \[
    S = \Sigma_g \setminus \left( \Cap(\delta_1)\cup \cdots \cup
    \Cap(\delta_{k-1}) \right)
    .\]
    If \(g(S) \geq 3\), then for separating simple closed curves
    \(\theta_1, \dots, \theta_m\) on \(S\), the following conditions are
    equivalent:
    \begin{itemize}
        \item In \(\B_2'\) we have
        \[
        \sigma(T_{\theta_1}) + \cdots + \sigma(T_{\theta_m}) \in \langle
        \sigma(T_{\delta_1}), \dots, \sigma(T_{\delta_{k-1}}) \rangle
        .\]
        \item In \(H_k(\I_g)\) we have
        \[
        \A(T_{\theta_1}, T_{\delta_1}, \dots, T_{\delta_{k-1}}) + \cdots +
        \A(T_{\theta_m}, T_{\delta_1}, \dots, T_{\delta_{k-1}}) = 0
        .\]
    \end{itemize}
\end{proposition}

\begin{proof}
    The proof of the proposition proceeds in an identical fashion to the proof
    of~\cite[Proposition~7.1]{vladimirov}. First, assume that 
    \[
    \A(T_{\theta_1}, T_{\delta_1}, \dots, T_{\delta_{k-1}}) + \cdots +
    \A(T_{\theta_m}, T_{\delta_1}, \dots, T_{\delta_{k-1}}) = 0
    .\] 
    Applying \(\sigma_k\) then yields  
    \begin{align*}
        0 &= \sigma_k(
        \A(T_{\theta_1}, T_{\delta_1}, \dots, T_{\delta_{k-1}}) + \cdots +
        \A(T_{\theta_m}, T_{\delta_1}, \dots, T_{\delta_{k-1}})) \\
        &= 
        \sigma_k(\A(T_{\theta_1} \cdots T_{\theta_m}, T_{\delta_1}, \dots,
        T_{\delta_{k-1}})) \\
        &= 
        \sigma(T_{\theta_1} \cdots T_{\theta_m}) \wedge \sigma(T_{\delta_1})
        \wedge \cdots \wedge \sigma(T_{\delta_{k-1}})
    .\end{align*}
    It follows that
    \[
    \sigma(T_{\theta_1}) + \cdots + \sigma(T_{\theta_m}) \in \langle
    \sigma(T_{\delta_1}), \dots, \sigma(T_{\delta_{k-1}}) \rangle
    ,\]
    as desired.

    Conversely, let us prove the other direction. By adding the curves
    \(\delta_1, \dots, \delta_{k-1}\) to the curves \(\theta_1, \dots,
    \theta_m\) if necessary, we may assume that
    \[
    \sigma(T_{\theta_1} \cdots T_{\theta_m}) = \sigma(T_{\theta_1}) + \cdots + \sigma(T_{\theta_m}) = 0
    .\]
    This modification does not alter the equality to be proved, as the abelian
    cycle vanishes whenever two entries coincide.

    For \(f = T_{\theta_1} \cdots T_{\theta_m}\), we have \(\tau_{\Sigma_g}(f)
    = 0\) and \(\sigma_{\Sigma_g}(f) = 0\). By
    Theorem~\ref{thm:torelli-abel-vanish}, it follows that \(f \in [\I(S),
    \I(S)]\); that is,
    \[
    f = [g_1, h_1] \cdots [g_{\ell}, h_{\ell}]
    ,\]
    where the elements \(g_j, h_j \in \I_g\) are supported on \(S\) for \(j =
    1, \dots, \ell\). This implies that
    \[
    \A(f, T_{\delta_1}, \dots, T_{\delta_{k-1}}) = \sum_{j=1}^{\ell} \A([g_j,
    h_j], T_{\delta_1}, \dots, T_{\delta_{k-1}}) = 0
    ,\]
    since each \(g_j\) and \(h_j\) commutes with \(T_{\delta_1}, \dots, T_{\delta_{k-1}}\).
\end{proof}

\subsection{Linear algebraic viewpoint}

Let \(\delta_1, \dots, \delta_k\) be an admissible partition of \(\Sigma_g\)
where each curve is outermost. This partition induces a splitting of
\(H_1(\Sigma_g)\) with respect to the intersection form:
\[
H_1(\Sigma_g) = U_1 \oplus \dots \oplus U_k \oplus U
,\]
where \(U_i = H_1(\Cap(\delta_i))\) for \(i = 1, \dots, k\), and \(U =
H_1(\Sigma_g \setminus \left( \Cap(\delta_1) \cup \dots \cup \Cap(\delta_k)
\right))\). We can therefore write the abelian cycle \(\A(T_{\delta_1},
\dots, T_{\delta_k})\) as \(\A(U_1, \dots, U_k)\); this is well-defined
by the following lemma.

\begin{lemma}
    The abelian cycle \(\A(T_{\delta_1}, \dots, T_{\delta_k})\) is uniquely determined
    by the collection \((U_1, \dots, U_k)\).
\end{lemma}

\begin{proof}
    Let \(\delta_1', \dots, \delta_k'\) be another admissible partition inducing
    the same splitting of \(H_1(\Sigma_g)\). Then there exists an element
    \(f \in \I_g\) mapping \(\delta_i\) to \(\delta_i'\) for all \(i = 1, \dots, k\).
    Consequently,
    \[
    \A(T_{\delta_1}, \dots, T_{\delta_k}) = f_{*} (\A(T_{\delta_1}, \dots, T_{\delta_k})) =
    \A(T_{\delta_1'}, \dots, T_{\delta_k'})
    ,\]
    which completes the proof.
\end{proof}

\subsection{Linear algebraic reformulation}

We begin by restating a special case of
Proposition~\ref{prop:key-relation-k} in terms of linear algebra.

\begin{proposition} \label{prop:key-rel-k-lin-alg}
    Let \(g \geq 4\) and \(2 \leq k \leq g-2\). Let \(V_1, \dots, V_n, U_1, \dots, U_{k-1}
    \subset H_1(\Sigma_g)\) be rank-\(2\) symplectic submodules of
    \(H_1(\Sigma_g)\) such that:
    \begin{enumerate}
        \item \(U_i \perp U_j\) for all \(i \neq j\);
        \item \(V_\ell \perp U_j\) for all \(\ell, j\).
    \end{enumerate}
    Then the following assertions are equivalent:
    \begin{itemize}
        \item In \(\B_2'\) we have
        \[
        \sigma(V_1) + \dots + \sigma(V_n) \in
        \langle \sigma(U_1), \dots, \sigma(U_{k-1}) \rangle
        .\] 
        
        \item In \(H_k(\I_g)\) we have
        \[
        \A(V_1, U_1, \dots, U_{k-1}) + \dots +
        \A(V_n, U_1, \dots, U_{k-1}) = 0
        .\] 
    \end{itemize}
\end{proposition}

\begin{corollary} \label{cor:change-fst}
    Let \(g \geq 4\) and \(2 \leq k \leq g-2\). Suppose that \(V, U, W_1, \dots,
    W_{k-1} \subset H_1(\Sigma_g)\) are rank-\(2\) symplectic submodules such
    that:
    \begin{enumerate}
        \item \(W_i \perp W_j\) for all \(i \neq j\);
        \item \(V, U \subset (W_1 \oplus \dots \oplus W_{k-1})^{\perp}\);
        \item \(\bV = \bU\) (that is \(\bmod 2\) reductions of \(V\) and \(U\)
        coincide).
    \end{enumerate}
    Then
    \[
    \A(V, W_1, \dots, W_{k-1}) = \A(U, W_1, \dots, W_{k-1})
    .\] 
\end{corollary}

\begin{proof}
    Since \(\bU = \bV\), we have \(\sigma(V) = \sigma(U)\). The corollary is
    then a direct consequence of Proposition~\ref{prop:key-rel-k-lin-alg}.
\end{proof}

    \section{A general equality criterion for abelian cycles}
\label{sec:sigma-k-ab-cyc}

The primary goal of this section is to establish
Proposition~\ref{prop:inj-k-2}, which generalizes
\cite[Proposition~7.1]{vladimirov}. The proof relies on a key technical
result, Lemma~\ref{lem:p-set-1-subspace}. In this section, we state this lemma
and use it to prove Proposition~\ref{prop:inj-k-2}, deferring its proof to
Section~\ref{sec:proof-of-auxiliary-lemma}.

Recall that we use the following convention: for symplectic
submodules \(V, U, \dots \) of \(H_1(\Sigma_g)\), their reductions modulo
\(2\) are denoted by \(\bV, \bU, \dots\).

\begin{proposition} \label{prop:inj-k-2}
    Let \(g \geq 4\) and \(2 \leq k \leq g-2\). Suppose that \(U_1, \dots, U_k,
    V_1, \dots, V_k \subset H_1(\Sigma_g)\) are rank-\(2\) symplectic submodules such that:
    \begin{enumerate}
        \item \(U_i \perp U_j\) for all \(i \neq j\);
        \item \(V_i \perp V_j\) for all \(i \neq j\);
        \item \(\bU_i = \bV_i\) for all \(i = 1, \dots, k\).
    \end{enumerate}
    Then 
    \[
    \A(V_1, \dots, V_k) = \A(U_1, \dots, U_k)
    .\] 
\end{proposition}

Now Theorem~\ref{mainthm:fin-dim} follows immediately from
Proposition~\ref{prop:inj-k-2}.

\begin{proof}[Proof of Theorem~\ref{mainthm:fin-dim}]
    By the linear-algebraic formulation of \cite[Theorem~1.1]{vladimirov}, the
    \(\Z/2\Z\)-vector space \(H_k^{\mathrm{ab,sep}}(\I_g)\) is generated by
    the abelian cycles \(\A(U_1, \dots, U_k)\) where \(U_1, \dots, U_k \subset
    H_1(\Sigma_g)\) are rank-\(2\) symplectic submodules. Since there are only
    finitely many \(2\)-dimensional subspaces in \(H_1(\Sigma_g;
    \Z/2\Z)\), Proposition~\ref{prop:inj-k-2} ensures that the total number
    of distinct abelian cycles \(\A(U_1, \dots, U_k)\) is finite.
\end{proof}

To prove Proposition~\ref{prop:inj-k-2}, we will need the following
auxiliary Lemma~\ref{lem:p-set-1-subspace}.

\begin{lemma} \label{lem:p-set-1-subspace}
    Let \(g \geq 4\) and \(2 \leq k \leq g-2\). Suppose that \(X_1, \dots,
    X_p, U_1, \dots, U_{k-p},
    V_1 \subset H_1(\Sigma_g)\) are rank-\(2\) symplectic submodules such that
    \begin{enumerate}
        \item \(X_i \perp X_j\) and \(U_i \perp U_j\) for all \(i \neq j\);
        \item \(X_i \perp V_1\) and \(X_i \perp U_{\ell}\) for all \(i, \ell\);
        \item \(\bU_1 = \bV_1\).
    \end{enumerate}
    Then there exist rank-\(2\) symplectic submodules \(W_2, \dots, W_{k-p}
    \subset \left( V_1 \oplus X_1 \oplus \dots \oplus X_p \right)^{\perp}\)
    such that 
    \begin{itemize}
        \item \(W_i \perp W_j\) for all \(i \neq j\);
        \item \(\bW_i = \bU_i\) for \(i = 2, \dots, k-p\);
        \item \(\A(X_1, \dots, X_p, V_1, W_2, \dots, W_{k-p}) = \A(X_1, \dots,
        X_p, U_1, \dots, U_{k-p})\).
    \end{itemize}
\end{lemma}

Postponing the proof of Lemma~\ref{lem:p-set-1-subspace} to
Section~\ref{sec:proof-of-auxiliary-lemma}, we now demonstrate how to derive
Proposition~\ref{prop:inj-k-2}. To cleanly isolate the inductive machinery and
streamline the exposition, we first establish the following intermediate
lemma, which encapsulates the core inductive step required for the
proposition.

\begin{lemma}\label{lem:m-inj-k-2}
    Let \(g \geq 4\) and \(2 \leq k \leq g-2\). For an integer \(1 \leq n
    \leq k\) suppose that \(U_1, \dots, U_k,
    V_1, \dots, V_n \subset H_1(\Sigma_g)\) are rank-\(2\) symplectic submodules such that:
    \begin{enumerate}
        \item \(U_i \perp U_j\) for all \(i \neq j\);
        \item \(V_i \perp V_j\) for all \(i \neq j\);
        \item \(V_i \perp U_j\) for \(i = 1, \dots, n\) and \(j = n+1, \dots,
        k\);
        \item \(\bU_i = \bV_i\) for \(i = 1, \dots, n\).
    \end{enumerate}
    Then 
    \[
    \A(V_1, \dots, V_n, U_{n+1}, \dots, U_k) = \A(U_1, \dots, U_k)
    .\] 
\end{lemma}

\begin{proof}
    We prove the lemma by induction on \(n\).

    \emph{Base case.} For \(n = 1\), the statement follows from
    Lemma~\ref{lem:p-set-1-subspace} with \(p = 0\).

    \emph{Induction step.} Assume that the statement holds for all \(n
    \leq \ell-1\), and consider the case \(n = \ell\). This means that
    \[
    \A(V_1, \dots, V_{\ell-1}, U_{\ell}, \dots, U_k) = \A(U_1,
    \dots, U_k)
    ,\] 
    whenever \(\bV_i = \bU_i\) for \(i = 1, \dots, \ell-1\).

    Applying Lemma~\ref{lem:p-set-1-subspace} with \(n = \ell-1\) to the
    abelian cycle \(\A(V_1, \dots, V_{\ell-1}, U_{\ell}, \dots, U_k)\) yields
    \[
    \A(V_1, \dots, V_{\ell-1}, U_{\ell}, \dots, U_k) =
    \A(V_1, \dots, V_{\ell-1}, V_{\ell}, U_{\ell+1}, \dots, U_k)
    ,\]
    whenever \(\bV_{\ell} = \bU_{\ell}\). The lemma follows.
\end{proof}

Proposition~\ref{prop:inj-k-2} follows immediately from
Lemma~\ref{lem:m-inj-k-2} by taking \(n = k\).

    \section{Proof of Lemma~\ref{lem:p-set-1-subspace}}
\label{sec:proof-of-auxiliary-lemma}

In this section, we prove Lemma~\ref{lem:p-set-1-subspace}. The argument
concentrates on the core case \(p = 0\), as the general case \(p > 0\) reduces
to the same machinery within an orthogonal complement.

\proofoutline{Lemma~\ref{lem:p-set-1-subspace}}
The key idea for \(p = 0\) is to replace a component \(U_1\)
with \(V_1\) (where \(\bU_1 = \bV_1\)). We achieve this in two stages. First,
we apply Lemma~\ref{lem:set-x-fst} to bring a chosen primitive element \(x \in
V_1\) into \(U_1\). Second, assuming \(x \in U_1\), we adjust the remaining
components of the abelian cycle until \(U_1 = V_1\)
(see Lemma~\ref{lem:set-v-fst}).
Both of these steps are driven by Lemma~\ref{lem:change-2-spaces}, which
allows us to simultaneously modify pairs of components in an abelian cycle
while keeping the other components fixed.

\begin{lemma} \label{lem:set-v-fst}
    Let \(g \geq 4\) and \(2 \leq k \leq g-2\). Suppose that \(V_1, U_1, \dots,
    U_k \subset H_1(\Sigma_g)\) are rank-\(2\) symplectic submodules with
    \(\bV_1 = \bU_1\), and let \(x \in V_1\) be a primitive element. Then for
    any integer \(1 \leq m \leq k\) there exist rank-\(2\) symplectic 
    submodules \(\widetilde{U}_1, \dots, \widetilde{U}_k \subset
    H_1(\Sigma_g)\) such that
    \begin{enumerate}
        \item \(\A(\widetilde{U}_1, \dots, \widetilde{U}_k) = \A(U_1, \dots,
        U_k)\);
        \item \(\widetilde{\bU}_i = \bU_i\) for all \(i = 1, \dots, k\);
        \item \(x \in V_1 \cap \widetilde{U}_1\);
        \item \(V_1 \subset \widetilde{U}_1 \oplus \dots \oplus \widetilde{U}_m\).
    \end{enumerate}
\end{lemma}

\begin{proof}[Proof of Lemma~\ref{lem:p-set-1-subspace}]
    As noted above, the case when \(p = 0\) follows directly from
    Lemma~\ref{lem:set-v-fst} by setting \(m = 1\).
    For \(p > 0\), the argument proceeds identically,
    except that all considerations are carried out within the orthogonal
    complement \(\left( X_1 \oplus \dots \oplus X_p \right)^{\perp}\) rather
    than \(H_1(\Sigma_g)\).
\end{proof}

Now we prove Lemma~\ref{lem:set-v-fst}. We begin with
the following two auxiliary lemmas.

\begin{lemma} \label{lem:change-2-spaces}
    Let \(g \geq 4\) and \(2 \leq k \leq g-2\). Suppose that \(X_1, X_2, Y_1,
    Y_2, U_1, \dots, U_{k-2} \subset H_1(\Sigma_g)\) are rank-\(2\) symplectic
    submodules such that:
    \begin{enumerate}
        \item \(U_i \perp U_j\) for all \(i \neq j\);
        \item \(X_1, X_2, Y_1, Y_2 \subset \left( U_1 \oplus \dots \oplus
        U_{k-2} \right)^{\perp}\);
        \item \(X_1 \perp X_2\) and \(Y_1 \perp Y_2\);
        \item \(\bX_i = \bY_i\) for \(i = 1,2\).
    \end{enumerate}
    Then 
    \[
    \A(X_1, X_2, U_1, \dots, U_{k-2}) = \A(Y_1, Y_2, U_1, \dots, U_{k-2})
    .\] 
\end{lemma}

\begin{proof}
    By Proposition~\ref{prop:mod-2-sigma}, we have \(\sigma(X_i) =
    \sigma(Y_i)\) for \(i = 1, 2\). The argument then proceeds step-by-step as
    in the proof of~\cite[Proposition~7.1]{vladimirov}, except that all
    considerations take place inside 
    \(\left( U_1 \oplus \dots \oplus U_{k-2} \right)^{\perp}\) rather than
    \(H_1(\Sigma_g)\).
\end{proof}

\begin{lemma} \label{lem:set-x-fst}
    Let \(g \geq 4\) and \(2 \leq k \leq g-2\). Suppose that \(U_1, \dots,
    U_k \subset H_1(\Sigma_g)\) are rank-\(2\) symplectic submodules and
    primitive \(x
    \in H_1(\Sigma_g)\) such that \(\bx \in \bU_1\).
    Then for any integer \(1 \leq m \leq k\) there exist rank-\(2\) symplectic 
    submodules \(\widetilde{U}_1, \dots, \widetilde{U}_k \subset
    H_1(\Sigma_g)\) such that 
    \begin{enumerate}
        \item \(\A(\widetilde{U}_1, \dots, \widetilde{U}_k) = \A(U_1, \dots,
        U_k)\);
        \item \(x \in \widetilde{U}_1 \oplus \dots \oplus \widetilde{U}_m\);
        \item \(\widetilde{\bU}_i = \bU_i\) for \(i = 1, \dots, k\).
    \end{enumerate}
\end{lemma}

\begin{proof}
    We prove the lemma by backward induction on \(m\).

    \emph{Base case}. For \(m = k\), choose a
    symplectic basis \(\{a_j, b_j\}\) for \(U_j\) for \(j = 1, \dots, k\) and
    \(a_{k+1} \in \left( U_1 \oplus \dots \oplus U_k \right)^{\perp}\) such
    that
    \[
    x = (2\zeta_1 + 1) a_1 + 2 \zeta_2 a_2 + \dots + 2
    \zeta_k a_k + 2 \zeta_{k+1} a_{k+1}
    .\] 

    We can choose a primitive representative of \((2\zeta_1 + 1) a_1 + 2 \zeta_{k+1} a_{k+1}\):
    \[
    a_1' = (2 \alpha_1 + 1) a_1 + 2 \alpha_{k+1} a_{k+1}
    ,\] 
    with \(\gcd (2 \alpha_1 + 1, 2 \alpha_{k+1}) = 1\). Then there exist 
    \(\beta_1, \beta_{k+1} \in \Z\) such that
    \[
    (2 \alpha_1 + 1)(2 \beta_1 + 1) + 4 \alpha_{k+1} \beta_{k+1} = 1
    .\] 
    Set
    \[
    b_1' = (2 \beta_1 + 1) b_1 + 2 \beta_{k+1} b_{k+1}
    .\] 
    Then \(a_1' \cdot b_1' = 1\) and for \(U_1' = \langle a_1', b_1' \rangle\)
    we have \(\bU_1' = \bU_1\). By Corollary~\ref{cor:change-fst}, we have
    \[
    \A(U_1, U_2, \dots, U_k) = \A(U_1', U_2, \dots, U_k)
    .\] 
    Setting \(\widetilde{U}_1 = U_1', \widetilde{U}_i =
    U_i\) for \(i = 2, \dots, k\) yields the lemma.

    \emph{Induction step}. Assume that the statement holds for all \(m \geq
    \ell\), and consider the case \(m = \ell - 1\). Choose a symplectic basis
    \(\{a_i, b_i\}\) for \(U_i\) for \(i = 1, \dots, \ell\) such that
    \[
    x = (2\zeta_1 + 1) a_1 + 2 \zeta_2 a_2 + \dots + 2 \zeta_\ell a_\ell
    .\]
    As in the base case, choose primitive representative of \((2 \zeta_1 + 1)
    a_1 + 2 \zeta_\ell a_\ell\):
    \[
    a_1' = (2 \alpha_1 + 1) a_1 + 2 \alpha_\ell a_\ell
    \]
    and consider
    \[
    b_1' = (2 \beta_1 + 1) b_1 + 2 \beta_\ell b_\ell
    \]
    such that \(a_1' \cdot b_1' = 1\). Then for \(U_1' = \langle a_1', b_1'
    \rangle\) we have \(\bU_1' = \bU_1\). Choose a rank-\(2\) symplectic
    submodule \(U_\ell' \subset H_1(\Sigma_g)\) such that
    \[
    U_1' \oplus U_\ell' = U_1 \oplus U_\ell
    .\] 
    Then \(\bU_\ell' = \bU_\ell\). By Lemma~\ref{lem:change-2-spaces} it
    follows that
    \[
    \A(U_1, \dots, U_\ell, \dots U_k) = \A(U_1', \dots, U_\ell', \dots, U_k)
    .\]
    Setting \(\widetilde{U}_1 = U_1', \widetilde{U}_\ell = U_\ell',
    \widetilde{U}_i = U_i\) for \(i \neq 1, \ell\) yields the lemma.
\end{proof}

Now we finally prove Lemma~\ref{lem:set-v-fst}.

\begin{proof}[Proof of Lemma~\ref{lem:set-v-fst}]
    Choose \(y \in V_1\) such that \(\{x, y\}\) is a symplectic basis for
    \(V_1\). We prove the lemma by backward induction on \(m\).

    \emph{Base case}. For \(m = k\), by Lemma~\ref{lem:set-x-fst}, we
    may assume that \(U_1, \dots, U_k\) are such that \(x \in U_1\).
    Choose a symplectic basis \(\{a_i, b_i\}\) for \(U_i\) for \(i = 1,
    \dots, k\) such that \(x = a_1\) and
    \[
    y = (2 \mu_1 + 1) b_1 + 2 \lambda_1 a_1 + 2 \mu_2 b_2 + \dots + 2 \mu_k
    b_k + 2 \mu_{k+1} b_{k+1}
    .\] 
    Let \(y' = y - 2 \mu_{k+1} b_{k+1}\) and \(V' = \langle x, y'
    \rangle\). Then \(\bV' = \bV = \bU_1 = \bU_1'\) and \(V' \subset \left(
    U_2 \oplus \dots \oplus U_k \right)^{\perp}\). By
    Corollary~\ref{cor:change-fst} we have
    \[
    \A(V', U_2, \dots, U_k) = \A(U_1, U_2, \dots, U_k)
    .\] 
    Thus, setting \(\widetilde{U}_1 = V', \widetilde{U}_i = U_i\) for \(i =
    2, \dots, k\) the lemma follows.

    \emph{Induction step}. Assume that the statement holds for all \(m \geq
    \ell\), and consider the case \(m = \ell - 1\). As in the base case, by
    Lemma~\ref{lem:set-x-fst}, we may assume that \(U_1, \dots, U_k\) are such
    that \(x \in U_1\). Choose a symplectic basis \(\{a_i, b_i\}\) for \(U_i\)
    for \(i = 1, \dots, \ell\) such that \(x = a_1\) and
    \[
    y = (2 \mu_1 + 1) b_1 + 2 \lambda_1 a_1 + 2 \mu_2 b_2 + \dots + 2 \mu_\ell
    b_\ell
    .\]
    Since \(x \cdot y = 1\), we have \(\mu_1 = 0\). Let \(y' = y - 2
    \lambda_1 a_1 = y - 2 \lambda_1 x \in V\). Then \(\{x, y'\}\) is
    symplectic basis for \(V\) and
    \[
    y' = b_1 + 2 \mu_2 b_2 + \dots + 2 \mu_\ell b_\ell
    .\]
    Let
    \begin{align*}
        b_1' &= b_1 + 2 \mu_\ell b_\ell \\
        a_\ell' &= a_\ell - 2 \mu_\ell a_1
    .\end{align*}
    Let \(U_1' = \langle a_1, b_1' \rangle\) and \(U_\ell' = \langle a_\ell',
    b_\ell \rangle\). Then \(U_1' \perp U_\ell'\) and \(U_1', U_\ell' \perp
    U_j\) for \(j \neq 1, \ell\). By Lemma~\ref{lem:change-2-spaces} we have
    \[
    \A(U_1', U_2, \dots, U_{\ell-1}, U_\ell', U_{\ell+1}, \dots, U_k) =
    \A(U_1, \dots, U_k)
    .\] 
    We also have \(y' \in U_1' \oplus U_2 \oplus \dots \oplus U_{\ell-1}\) and
    thus \(V \subset U_1' \oplus U_2 \dots \oplus U_{\ell-1}\).
    Moreover, we also have \(\bU_1' = \bU_1\) and \(\bU_\ell' = \bU_\ell\). The
    lemma now follows by setting \(\widetilde{U}_1 = U_1', \widetilde{U}_\ell =
    U_\ell', \widetilde{U}_j = U_j\) for all \(j \neq 1, \ell\).
\end{proof}

    \section{Symplectic submodules under the Birman--Craggs--Johnson homomorphism}
\label{sec:symp-submod-BCJ}

In this section, we establish a few technical results concerning the images of
symplectic submodules under the Birman--Craggs--Johnson homomorphism. We begin
with the following lemma.

\begin{lemma} \label{lem:BCJ-gen-1}
    Let \(U, V \subset H_1(\Sigma_g)\) be symplectic submodules of rank
    \(2\) and \(r\), respectively, satisfying \(2 < r\) and \(2 + r < 2g\).
    Then \(\sigma(U) \neq \sigma(V)\) in \(\B_2'\).
\end{lemma}

\begin{proof}
    Since \(2 + r < 2g\), it suffices to prove that the inequality \(\sigma(U)
    \neq \sigma(V)\) holds in \(\B_2\). The assertion follows immediately from
    the fact that there exists an element \(x \in H_1(\Sigma_g)\) such that
    \(\overline{\bx} \, \sigma(U) \in \B_2\), whereas no such element
    exists for \(\sigma(V)\).

    Indeed, for \(\sigma(U)\), we may select a symplectic basis \(\{a_1,
    b_1\}\) for \(U\)
    and set \(x = a_1\). To see that no such element exists for \(\sigma(V)\), suppose
    on the contrary that such an \(x \in H_1(\Sigma_g)\) is given. Fixing a
    symplectic basis \(\{a_1, b_1, \dots, a_r, b_r\}\) for \(V\), we have
    \[
    \overline{\bx} \left( \overline{\ba}_1 \overline{\bb}_1 + \dots +
    \overline{\ba}_r \overline{\bb}_r \right) \in \B_2
    .\]
    Denoting by \(\bx_1\) the projection of \(\bx\) onto \(\langle \ba_1, \bb_1 \rangle^{\perp}\),
    this leads to
    \[
    \overline{\bx}_1 \overline{\ba}_1 \overline{\bb}_1 + \overline{\bx}
    \left(\overline{\ba}_2 \overline{\bb}_2 + \dots + \overline{\ba}_r
    \overline{\bb}_r\right) \in \B_2
    .\]
    However, since \(\overline{\bx}_1 \overline{\ba}_1 \overline{\bb}_1 \notin \B_2\)
    and none of the remaining monomials contain the product \(\overline{\ba}_1
    \overline{\bb}_1\), the entire sum cannot lie in \(\B_2\), which is
    impossible.
\end{proof}

\begin{corollary} \label{cor:sigma-rank}
    For any integer \(n\) satisfying \(2 \leq n \leq g-2\) and any rank-\(2\) 
    symplectic submodules \(U, V_1, \dots, V_n \subset H_1(\Sigma_g)\) such 
    that \(V_i \perp V_j\) for all \(i \neq j\), we have
    \[
    \sigma(U) \neq \sigma(V_1) + \dots + \sigma(V_n)
    .\] 
\end{corollary}

\begin{proof}
    Applying Lemma~\ref{lem:BCJ-gen-1} to the
    symplectic submodules \(U, V_1 \oplus \dots \oplus V_n \subset
    H_1(\Sigma_g)\) yields the assertion.
\end{proof}

\begin{corollary} \label{cor:wedge-sigma}
    Suppose that \(U_1, \dots, U_k, V_1, \dots, V_k \subset H_1(\Sigma_g)\)
    (with \(k \leq g-2\)) are rank-\(2\) symplectic submodules such that 
    \begin{itemize}
        \item \(U_i \perp U_j\) for all \(i \neq j\);
        \item \(V_i \perp V_j\) for all \(i \neq j\).
    \end{itemize}
    Then equality
    \[
    \sigma(U_1) \wedge \dots \wedge \sigma(U_k) = \sigma(V_1) \wedge \dots
    \wedge \sigma(V_k)
    \] 
    implies
    \[
    \{\bU_1, \dots, \bU_k\} = \{\bV_1, \dots, \bV_k\}
    .\] 
\end{corollary}

\begin{proof}
    The equality
    \[
    \sigma(U_1) \wedge \dots \wedge \sigma(U_k) = \sigma(V_1) \wedge \dots
    \wedge \sigma(V_k)
    \] 
    implies that for some indices \(i_1, \dots, i_\ell\) we have
    \[
    \sigma(V_1) = \sigma(U_{i_1}) + \dots + \sigma(U_{i_\ell})
    .\] 
    By Corollary~\ref{cor:sigma-rank}, it follows that \(\ell = 1\), which means
    \(\sigma(V_1) = \sigma(U_{i_1})\). Repeating this argument for the
    remaining components \(V_2, \dots, V_k\) yields
    \[
    \{\sigma(V_1), \dots, \sigma(V_k)\} = \{\sigma(U_1), \dots, \sigma(U_k)\}
    .\] 
    Without loss of generality, we may assume that \(\sigma(V_j) =
    \sigma(U_j)\) for all \(j =
    1, \dots, k\). Applying Proposition~\ref{prop:mod-2-sigma} then implies
    that \(\bV_j = \bU_j\) for each \(j = 1, \dots, k\), and the corollary
    follows.
\end{proof}

    \section{Proof of Theorem~\ref{mainthm:BCJ-inj}} \label{sec:sigma2-inj}

In this section, we prove Theorem~\ref{mainthm:BCJ-inj}.

\proofoutline{Theorem~\ref{mainthm:BCJ-inj}}
The proof of Theorem~\ref{mainthm:BCJ-inj} proceeds as follows. First, we
collect the necessary facts about abelian cycles and translate them into
statements about cycles whose components are \(2\)-dimensional symplectic
subspaces of \(H_1(\Sigma_g; \Z/2\Z)\). Second, we replace the domain of
\(\sigma_k\) with an isomorphic \(\Z/2\Z\)-vector space expressed in terms of
these symplectic subspaces, yielding a new homomorphism
\(\widetilde{\sigma}_k\). Finally, we prove that \(\widetilde{\sigma}_k\) is
injective, which in turn implies the injectivity of \(\sigma_k\).

\subsection{Abelian cycles setup}

In this subsection, we summarize the necessary background regarding linear 
subspaces and abelian cycles. We begin by recalling the following 
result from~\cite{vladimirov}.

\begin{proposition} \label{prop:gen-1-generation}
    Let \(g \geq 3\) and \(1 \leq k \leq g-1\). The \(\Z/2\Z\)-vector space
    \(H_k^{\mathrm{ab,sep}}(\I_g)\) is generated by the abelian cycles \(\A(U_1,
    \dots, U_k)\), where \(U_1, \dots, U_k \subset H_1(\Sigma_g)\) are
    rank-\(2\) symplectic submodules satisfying \(U_i \perp U_j\) for all \(i
    \neq j\).
\end{proposition}

We have the following proposition.

\begin{proposition} \label{prop:sigma-k-inj2}
    Let \(U_1, \dots, U_k, V_1, \dots, V_k \subset H_1(\Sigma_g)\) be
    rank-\(2\) symplectic submodules such that 
    \begin{itemize}
        \item \(U_i \perp U_j\) for all \(i \neq j\);
        \item \(V_i \perp V_j\) for all \(i \neq j\).
    \end{itemize}
    Then the equality
    \[
    \sigma_k(\A(V_1, \dots, V_k)) = \sigma_k(\A(U_1, \dots, U_k))
    \]
    implies that
    \[
    \A(V_1, \dots, V_k) = \A(U_1, \dots, U_k)
    .\]
\end{proposition}

\begin{proof}
    The equality \(\sigma_k(\A(V_1, \dots, V_k)) = \sigma_k(\A(U_1, \dots,
    U_k))\) yields
    \[
    \sigma(V_1) \wedge \dots \wedge \sigma(V_k) = \sigma(U_1) \wedge \dots \wedge \sigma(U_k)
    .\]
    Corollary~\ref{cor:wedge-sigma} implies that 
    \[
    \{\bV_1, \dots, \bV_k\} = \{\bU_1, \dots, \bU_k\}
    .\]

    Since the abelian cycle \(\A(V_1, \dots, V_k)\) has order~\(2\), it is
    invariant under permutations of its components. Therefore, without loss of
    generality, we may assume that \(\bV_j = \bU_j\) for all \(j = 1, \dots, k\).
    The proposition now follows from Proposition~\ref{prop:inj-k-2}.
\end{proof}

\subsection{Abelian cycles with components modulo 2}

In this subsection we will show how to define the abelian cycle \(\A(\bU_1,
\dots, \bU_k) \in H_k(\I_g)\), where \(\bU_1, \dots, \bU_k \subset
H_1(\Sigma_g; \Z/2\Z)\) are \(2\)-dimensional symplectic subspaces with
\(\bU_i \perp \bU_j\) for all \(i \neq j\).

We begin with the following standard lemma from linear algebra.

\begin{lemma} \label{lem:lift-mod-2}
    Let \(\bU_1, \dots, \bU_m\) be \(2\)-dimensional symplectic subspaces in
    \(H_1(\Sigma_g; \Z/2\Z)\) such that \(\bU_i \perp \bU_j\) for all \(i \neq
    j\). Then there exist rank-\(2\) symplectic submodules \(\widetilde{U}_1,
    \dots, \widetilde{U}_m \subset H_1(\Sigma_g)\) such that 
    \begin{itemize}
        \item \(\widetilde{U}_i \perp \widetilde{U}_j\) for all \(i \neq j\); and
        \item \(\widetilde{\bU}_j = \bU_j\) for \(j = 1, \dots, m\).
    \end{itemize}
    We call \(\widetilde{U}_1, \dots, \widetilde{U}_m\) a \emph{lift} of
    \(\bU_1, \dots, \bU_m\).
\end{lemma}

\begin{proposition} \label{prop:ab-cyc-mod-2}
    Let \(\bU_1, \dots, \bU_k \subset H_1(\Sigma_g; \Z/2\Z)\) be
    \(2\)-dimensional symplectic subspaces with \(\bU_i \perp \bU_j\) for all \(i
    \neq j\). Then the abelian cycle \(\A(\bU_1, \dots, \bU_k) \in
    H_k(\I_g)\), defined by setting \(\A(\bU_1, \dots, \bU_k) = \A(U_1, \dots,
    U_k)\) for a lift \((U_1, \dots, U_k)\) of \((\bU_1, \dots, \bU_k)\), is
    well-defined; that is, it does not depend on the choice of lift.
\end{proposition}

\begin{proof}
    By Corollary~\ref{lem:lift-mod-2}, there exist a lift \(U_1, \dots, U_k\)
    of \(\bU_1, \dots, \bU_k\). Then we set
    \[
    \A(\bU_1, \dots, \bU_k) = \A(U_1, \dots, U_k)
    .\] 
    The fact that \(\A(\bU_1, \dots, \bU_k)\) does not depend on the choice of
    the lift follows from Proposition~\ref{prop:sigma-k-inj2}. Indeed, for any
    other lift \(U_1', \dots, U_k'\) of \(\bU_1, \dots, \bU_k\) we have
    \[
    \sigma_k(\A(U_1, \dots, U_k)) = \sigma(\bU_1) \wedge \dots \wedge
    \sigma(\bU_k) = \sigma_k(\A(U_1', \dots, U_k')) 
    .\]
    Proposition~\ref{prop:sigma-k-inj2} then implies that 
    \[
    \A(U_1, \dots, U_k) = \A(U_1', \dots, U_k')
    ,\] 
    which completes the proof.
\end{proof}

Proposition~\ref{prop:ab-cyc-mod-2} implies that established properties of the
abelian cycles \(\A(U_1, \dots, U_k)\) carry over directly to the cycles
\(\A(\bU_1, \dots, \bU_k) \in H_k(\I_g)\). We summarize these results in the
following two propositions.

\begin{proposition} \label{prop:gen-1-generation-mod-2}
    Let \(g \geq 3\) and \(1 \leq k \leq g-1\). The \(\Z/2\Z\)-vector space
    \(H_k^{\mathrm{ab,sep}}(\I_g)\) is generated by the abelian cycles
    \(\A(\bU_1, \dots, \bU_k)\), where \(\bU_1, \dots, \bU_k \subset
    H_1(\Sigma_g; \Z/2\Z)\) are \(2\)-dimensional symplectic subspaces
    satisfying \(\bU_i \perp \bU_j\) for all \(i \neq j\).
\end{proposition}

\begin{proposition} \label{prop:ab-cyc-relations-mod-2}
    Let \(g \geq 4\) and \(2 \leq k \leq g-2\). We have the following
    relations between abelian cycles \(\A(\bU_1, \dots, \bU_k) \in
    H_k(\I_g)\), where \(\bU_1, \dots, \bU_k \subset H_1(\Sigma_g; \Z/2\Z)\) are \(2\)-dimensional
    symplectic subspaces satisfying \(\bU_i \perp \bU_j\) for all \(i \neq j\):
    \begin{enumerate}
        \item \(\A(\bU_{\pi(1)}, \dots, \bU_{\pi(k)}) = \A(\bU_1, \dots, \bU_k)\) for
        any permutation \(\pi\) of \(\{1, \dots, k\}\);
        \item \(\A(\bU, \bU, \bU_1, \dots, \bU_{k-2}) = 0\);
        \item \(2 \A(\bU_1, \dots, \bU_k) = 0\);
        \item \(\A(\bU_1, \bV_1, \dots, \bV_{k-1}) + \dots + \A(\bU_n, \bV_1, \dots,
        \bV_{k-1}) = 0\) if \(\sigma(\bU_1) +
        \dots + \sigma(\bU_n) = 0\), where \(\bU_1, \dots, \bU_n \subset (\bV_1 \oplus
        \dots \oplus \bV_{k-1})^{\perp}\) are \(2\)-dimensional symplectic
        subspaces.
    \end{enumerate}
\end{proposition}

\subsection{Modified domain and the map \(\widetilde{\sigma}_k\)}

Let \(\Symp\) denote the set of \(2\)-dimensional symplectic subspaces of
\(H_1(\Sigma_g; \Z/2\Z)\). Consider the \(\Z/2\Z\)-vector space \(\AA\) freely 
spanned by the formal basis \(\{ [\bU] \mid \bU \in \Symp \}\).

The Birman--Craggs--Johnson homomorphism induces a surjective homomorphism 

\[
\begin{tikzcd}[column sep = 0.7cm]
    \widetilde{\sigma} \colon \AA \arrow[r, two heads] &\B_2'
.\end{tikzcd}
\]
We set
\[
\CC = \langle [\bU_1] + \dots + [\bU_n] \mid \sigma(\bU_1) + \dots +
\sigma(\bU_n) = 0, \, \bU_1, \dots, \bU_n \in \Symp \rangle \subset \AA
.\]
Thus, we have
\[
\B_2' \cong \AA / \CC
.\] 

We have the following standard fact. 

\begin{fact*}
Let \(F\) be a module over a ring \(R\) and let \(L \subset F\) be a
submodule. For \(V = F/L\), we have
\[
\wedge^k \, V \cong F^{\otimes k} / \left( J_L(F) + S(F) \right)
,\]
where
\begin{align*}
    J_L(F) &= \sum_{i=1}^{k} F^{\otimes i-1} \otimes L \otimes F^{\otimes
    k-i}, \\
    S(F) &= \langle f_1 \otimes \dots \otimes f_k \mid f_m \in F, f_i = f_j \text{ for some
    } i, j \rangle 
.\end{align*}

\end{fact*}

It follows that
\[
\wedge^k \B_2' \cong \AA^{\otimes k}  / \left( J_{\CC}(\AA) +
\SS(\AA) \right)
,\] 
where 
\begin{align*}
    J_{\CC}(\AA) &= \sum_{i=1}^{k} \AA^{\otimes i-1} \otimes \CC \otimes \AA^{\otimes k-i}, \\
    \SS(\AA) &= \langle a_1 \otimes \dots \otimes a_k \mid a_m \in \AA, a_i =
    a_j \text{ for some } i, j \rangle 
.\end{align*}

In what follows, we prove the injectivity of the homomorphism
\[
\widetilde{\sigma}_k \colon H_k^{\mathrm{ab,sep}}(\I_g) \longrightarrow
\AA^{\otimes k} / \left( J_{\CC}(\AA) + \SS(\AA) \right)
,\] 
given by
\[
\A(\bU_1, \dots, \bU_k) \longmapsto [\bU_1] \otimes \dots \otimes [\bU_k]
.\]

\subsection{Proof of injectivity of \(\widetilde{\sigma}_k\)}

Let \(\FF_{\A}\) be the free \(\Z/2\Z\)-vector space spanned by the abelian
cycles \(\A(\bU_1, \dots, \bU_k)\), where \(\bU_1, \dots, \bU_k \in \Symp\)
satisfy \(\bU_i \perp \bU_j\) for all \(i \neq j\). Let \(\RR_{\A}\) denote
the kernel of the natural projection \(\FF_{\A} \to
H_k^{\mathrm{ab,sep}}(\I_g)\), i.e., \(H_k^{\mathrm{ab,sep}}(\I_g) \cong
\FF_{\A} / \RR_{\A}\).

To establish the
injectivity of \(\widetilde{\sigma}_k\), it therefore suffices to show that
under the homomorphism \(\FF_{\A} \to \AA^{\otimes k}\) given by
\[
\A(\bU_1, \dots, \bU_k) \longmapsto [\bU_1] \otimes \dots \otimes [\bU_k]
,\]
the image of \(\RR_{\A}\) in \(\AA^{\otimes k}\) contains the subspace
\(J_{\CC}(\AA) + \SS(\AA)\). The relations stated in
Proposition~\ref{prop:ab-cyc-relations-mod-2} yield exactly this inclusion,
from which injectivity follows.

\begin{proof}[Proof of Theorem~\ref{mainthm:BCJ-inj}]
    By construction, \(\sigma_k\) is the composition of
    \(\widetilde{\sigma}_k\) with the isomorphism
    \[
    \AA^{\otimes k}  / \left( J_{\CC}(\AA) + \SS(\AA) \right) \cong \wedge^k \B_2'
    .\]
    Therefore, the injectivity of \(\widetilde{\sigma}_k\) immediately implies
    that \(\sigma_k\) is injective.
\end{proof}

    \printbibliography
\end{document}